\documentclass[11pt]{amsart}

\usepackage[T1]{fontenc}
\usepackage[utf8]{inputenc}
\usepackage{lmodern}
\usepackage{amsmath,amssymb,amsthm,mathrsfs}
\usepackage{enumitem}
\usepackage{hyperref}
\usepackage[margin=1.15in]{geometry}
\usepackage{pdfsync}
\usepackage{xcolor}
\numberwithin{equation}{section}

\hypersetup{colorlinks=true,linkcolor=blue,citecolor=blue,urlcolor=blue}

\newtheorem{theorem}{Theorem}[section]
\newtheorem{proposition}[theorem]{Proposition}
\newtheorem{lemma}[theorem]{Lemma}
\newtheorem{corollary}[theorem]{Corollary}
\theoremstyle{definition}

\newtheorem{conjecture}[theorem]{Conjecture}
\theoremstyle{remark}
\newtheorem{remark}[theorem]{Remark}
\newtheorem{example}[theorem]{Example}

\newcommand{\C}{\mathbb C}

\newcommand{\norm}[1]{\left\lVert #1\right\rVert}

\newcommand{\Ree}{\operatorname{Re}}
\newcommand{\Imm}{\operatorname{Im}}
\def\i{\mathrm{i}}
\def\bS{\mathbb{T}}
\renewcommand{\Re}{\operatorname{Re}}
\renewcommand{\Im}{\operatorname{Im}}

\title[Area operators on Hardy spaces of Dirichlet series]{Area operators on Hardy spaces of Dirichlet series II: counterexamples and compactness criteria}

\author[J. Chen]{Jiale Chen}
\address{Jiale Chen, School of Mathematics and Statistics, Shaanxi Normal University, Xi'an 710119, China.}
\email{jialechen@snnu.edu.cn}

\author[M. Wang]{Maofa Wang}
\address{Maofa Wang, School of Mathematics and Statistics, Wuhan University, Wuhan 430072, China.}
\email{mfwang.math@whu.edu.cn}
	
\author[Z. Yuan]{Zixing Yuan}
\address{Zixing Yuan, Institute of Mathematics, Hebei University of Technology, Tianjin 300401, China.}
\email{zxyuan@hebut.edu.cn}

\thanks{Chen is supported by National Natural Science Foundation (No. 12501170), Young Talent Fund of Association for Science and Technology in Shaanxi (No. 20260506) and the Fundamental Research Funds for the Central Universities (No. GK202601003) of China. Wang is supported by National Natural Science Foundation (No. 12571145) of China.}

\date{\today}

\begin{document}

\begin{abstract}
We study area operators $\mathbb{A}_{\mu,l}$, $0<l<\infty$, induced by positive Borel measures on the right half-plane and acting on the Hardy spaces of Dirichlet series $\mathscr H^p$, $0<p<\infty$. We first disprove a conjecture proposed by the present authors in an earlier work by constructing a probability measure, valid for all $0<p,l<\infty$, for which the associated area operator is bounded although the measure fails the proposed Carleson conditions. We then characterize the boundedness and compactness of $\mathbb{A}_{\mu,l}$ on both $\mathscr H^p$ and the Hardy space $\mathscr H^p_0$ of Dirichlet series vanishing at $+\infty$, for every $0<p,l<\infty$, via integrability conditions for the associated cone integrals. In particular, the boundedness and compactness are equivalent for these operators. We also give direct proofs of compactness under vanishing Carleson measure and compact $H_{\i}^p$-Carleson embedding	hypotheses. As an application, we establish an estimate for the essential norms of bounded Volterra operators. This leads to a complete characterization of compactness for Volterra operators on $\mathscr{H}^p$ and provides a different proof of compactness for those with Dirichlet series symbols in $\operatorname{VMOA}(\mathbb C_0)$.
\end{abstract}

\maketitle

\section{Introduction}

Area operators on classical spaces of analytic functions have been studied
extensively.  Let $\mathbb D$ be the unit disk and $\mathbb T:=\partial\mathbb D$.
Given $\gamma>1$ and $\xi\in\mathbb T$, the non-tangential approach region
with vertex $\xi$ is defined by
\[
 \Gamma_\gamma(\xi)
 :=\{z\in\mathbb D:|\xi-z|<\gamma(1-|z|)\}.
\]
For a positive Borel measure $\mu$ on $\mathbb D$ and $0<l<\infty$, define
\[
 A_\mu^l f(\xi)
 :=\left(\int_{\Gamma_\gamma(\xi)}|f(z)|^l
 \frac{d\mu(z)}{1-|z|}\right)^{1/l},
 \qquad \xi\in\mathbb T.
\]
When $l=1$, we simply write $A_\mu=A_\mu^1$.  Cohn \cite{Cohn}
introduced the operator $A_\mu$ and proved that, for every $0<p<\infty$,
$A_\mu$ is bounded from the classical Hardy space $H^p(\mathbb D)$ into the Lebesgue space $L^p(\mathbb T)$ if and only if $\mu$
is a Carleson measure on $\mathbb D$.  Gong, Lou and Wu \cite{GongLouWu}
subsequently studied the more general cases
$A_\mu^l:H^p(\mathbb D)\to L^q(\mathbb T)$.  For $0<p\le q<\infty$,
the boundedness of $A_\mu^l:H^p(\mathbb D)\to L^q(\mathbb T)$ is equivalent to $\mu$ being a $\beta$-Carleson measure with
$\beta=1+l(1/p-1/q)$, that is,
$\mu(S(I))\lesssim |I|^\beta$ for every arc $I\subset\mathbb T$, where
$S(I)$ is the Carleson square generated by $I$ and $|I|$ denotes the length of $I$.  Gong, Lou and Wu also
treated several cases with $q<p$.  More recently, Lv and Pau
\cite{LvPau} gave a complete characterization of the boundedness of
$A_\mu^l:H^p(\mathbb D)\to L^q(\mathbb T)$ for all
$0<p,q,l<\infty$; in fact, their result is proved in the more general
setting of the unit ball.  For the diagonal target exponent $q=p$, compactness is described by vanishing Carleson measure conditions; see \cite{LiuLouZhao} and the references
therein.  In particular, when $p=q$, one has
\[
 A_\mu^l:H^p(\mathbb D)\to L^p(\mathbb T)
 \quad\hbox{is bounded}
 \quad\Longleftrightarrow\quad
 \mu\ \hbox{is a Carleson measure},
\]
whereas
\[
 A_\mu^l:H^p(\mathbb D)\to L^p(\mathbb T)
 \quad\hbox{is compact}
 \quad\Longleftrightarrow\quad
 \mu\ \hbox{is a vanishing Carleson measure}.
\]
Thus boundedness and compactness do not coincide in general in the
classical Hardy space setting, even when $p=l$.  Related area operators
on Bergman spaces and their weighted variants have also been studied;
see \cite{ArroussiTaskinenTongYuan,LvPau-JOT,LvPauWang,PangPeralaWang,PelaezRattyaSierra,WuBergman} and the references therein.

The situation for Hardy spaces of Dirichlet series turns out to be
substantially different.  In our previous work \cite{ChenWangYuan},
motivated by Cohn's construction, we introduced area operators in the
Dirichlet series setting.  Let $\mu$ be a positive Borel measure on
$\mathbb C_0:=\{s=\sigma+it:\sigma>0\}$ and let $0<l<\infty$.  For
$\tau\in\mathbb R$, put
\[
 \Gamma_\tau:=\{\sigma+it\in\mathbb C_0:|t-\tau|<\sigma\}.
\]
The associated area operator $\mathbb{A}_{\mu,l}$ is formally defined for Dirichlet series $f$ by
\[
 \mathbb{A}_{\mu,l}f(\chi,\tau)
 :=\left(\int_{\Gamma_\tau}|f_\chi(s)|^l
   \frac{d\mu(s)}{\Ree s}\right)^{1/l},
\]
where $f_\chi$ is the vertical limit of $f$. We proved in \cite{ChenWangYuan} that for $0<p,l<\infty$, an ordinary Carleson measure on $\mathbb C_0$ induces a bounded area operator $\mathbb{A}_{\mu,l}$ on the Hardy space $\mathscr H_0^p$ of Dirichlet series vanishing at $+\infty$, and that an $H_{\i}^p(\mathbb C_0)$-Carleson measure gives a sufficient condition for $\mathbb{A}_{\mu,l}$ to be bounded on the full Hardy space $\mathscr H^p$ of Dirichlet series.  This led us to ask in \cite[Conjecture 5.1]{ChenWangYuan} whether the corresponding Carleson measure conditions on the half-plane $\mathbb{C}_{1/2}$ are necessary. In this paper, we are going to disprove this conjecture and then characterize boundedness and compactness of $\mathbb A_{\mu,l}$ on both $\mathscr H^p$ and $\mathscr H_0^p$.

The first phenomenon established in the present paper is that the classical Carleson measure description is no longer valid for Hardy spaces of Dirichlet series.  We construct a probability measure, valid simultaneously for every $0<p,l<\infty$, for which the associated area operator is bounded although the measure fails the Carleson measure conditions proposed in \cite[Conjecture 5.1]{ChenWangYuan}.  Thus, unlike the classical Hardy space case, the usual Carleson measure condition is not necessary for boundedness of area operators on Hardy spaces of Dirichlet series.

The second difference concerns compactness. In the classical
	setting, as recalled above, boundedness corresponds to a Carleson
	measure condition while compactness requires its vanishing version,
	so the two properties do not coincide in general, even when $p=l$. In the setting of
	Hardy spaces of Dirichlet series, however, we prove that, for every
	$0<p,l<\infty$, $\mathbb A_{\mu,l}:\mathscr H^p\to L_{\i}^p$ is
	bounded if and only if it is compact, where the precise definition of
	the natural target space $L_{\i}^p$ is given in Section
	\ref{sec:preliminaries}. The same equivalence also holds on
	$\mathscr H_0^p$.  More precisely, on $\mathscr H^p$ and $\mathscr H_0^p$ these
	equivalent properties are characterized by
	\[
	\int_{\mathbb R}
	\left(\int_{\Gamma_\tau}n^{-l\Ree s}
	\frac{d\mu(s)}{\Ree s}\right)^{p/l}
	\frac{d\tau}{1+\tau^2}<\infty,
	\qquad n=1,\ 2,
	\]
	respectively; see Theorem~\ref{thm:general-exponents}.
	This gives another clear difference between area
	operators on classical Hardy spaces and those on Hardy spaces of
	Dirichlet series.
	The proof of this characterization treats a fixed measure
	independent of $\chi$.
	The cone weight can therefore be taken outside the Haar integral,
	after which Lemma~\ref{lem:cone-maximal} yields an estimate
	averaged over $\chi$ and expressed in terms of these cone integrals.
	If the measure also depends on $\chi$, this step no longer applies
	directly, since both the cone weight and $f_\chi$ depend on $\chi$.

For arbitrary $0<p,l<\infty$, we also give direct proofs of compactness on $\mathscr H_0^p$ for vanishing Carleson measures on $\mathbb C_0$, and on $\mathscr H^p$ when the embedding $H_{\i}^p(\mathbb C_0)\hookrightarrow L^p(\mathbb C_0,\mu)$ is compact. Although these two conclusions already follow from Theorems~\ref{thm:CYW} and~\ref{thm:general-exponents}, the direct proofs yield estimates that can be used for measures depending on $\chi$. The underlying one-variable Carleson or embedding estimates can be applied to each $f_\chi$ with its corresponding measure before Haar integration, provided the constants are independent of $\chi$ and the boundary or tail contributions tend to zero uniformly in $\chi$. 
In the classical Hardy space setting, the contribution from the region away from the boundary is typically handled by local uniform convergence on compact subsets of $\mathbb D$. 
In the present setting, the area operator involves the whole family of vertical limits $f_\chi$, $\chi\in\mathbb T^\infty$, and local uniform convergence does not directly give the required uniform control in $\chi$. We therefore split the measure into a part near the boundary and a part away from the boundary, and handle the latter by the compactness of horizontal translations on $\mathscr H^p$. The same argument applies to the classical setting, with horizontal translations replaced by the radial dilations $f(\cdot)\mapsto f(r\cdot)$, $0<r<1$. The method can also be adapted to Volterra operators, whose associated measures depend on \(\chi\). As an application, we establish an estimate for the essential norms of bounded Volterra operators on $\mathscr{H}^p$. This estimate allows us to obtain a complete characterization of compactness for Volterra operators on $\mathscr{H}^p$, and provide a different proof of the compactness for those with symbols in \(\operatorname{VMOA}(\mathbb C_0)\), first obtained by Brevig, Perfekt and Seip in \cite[Section~7]{BPSVolterra}.

The paper is organized as follows. Section
\ref{sec:preliminaries} contains some definitions and preliminary
results. In Section \ref{sec:counterexample}, we give a counterexample
that disproves \cite[Conjecture 5.1]{ChenWangYuan}. Section
\ref{sec:diagonal} is devoted to the characterization of boundedness
and compactness for $\mathbb A_{\mu,l}$ on the spaces
$\mathscr H_0^p$ and $\mathscr H^p$ for every $0<p,l<\infty$.
Section \ref{sec:compactness} gives direct proofs of compactness
on $\mathscr H_0^p$ for vanishing Carleson measures on $\mathbb C_0$,
and on $\mathscr H^p$ when the embedding
$H_{\i}^p(\mathbb C_0)\hookrightarrow L^p(\mathbb C_0,\mu)$ is compact.
{Section \ref{sec:volterra} studies essential norms and compactness of bounded Volterra operators}.

Throughout the paper, we write $A\lesssim B$ if $A\leq CB$ for some
inessential constant $C>0$, and $A\asymp B$ if $A\lesssim B\lesssim A$. We also write $C_{p,l}$ for positive constants depending on the parameters $p,l$.
For $\theta\in\mathbb R$, we set
$\mathbb C_\theta:=\{s\in\mathbb C:\Re(s)>\theta\}$.
If $f(s)=\sum_{n=1}^{\infty}a_nn^{-s}$ is a Dirichlet series, then
$f(+\infty)$ denotes its constant coefficient $a_1$.  For a positive Borel
measure $\mu$ on $\mathbb C_0$ and $\varepsilon>0$, we write
$\mu_\varepsilon:=\mu|_{\{0<\sigma<\varepsilon\}}$ and
$\mu^\varepsilon:=\mu|_{\{\sigma\ge\varepsilon\}}$. For a set $E\subset\C_0$, we use $\mathbf 1_E$ to denote its characteristic function.

\section{Definitions and preliminaries}\label{sec:preliminaries}

In this section, we recall some definitions and preliminary results that will be used throughout the paper.

Let $\bS^{\infty}$ be the countably infinite-dimensional torus carrying the
normalized Haar measure $m_{\infty}$, which coincides with the product of the
normalized Lebesgue measure $m$ on the unit circle $\bS$.
We identify
$\bS^{\infty}$ with the group of characters $\chi$ on the positive integers
satisfying
\begin{enumerate}[label=\textup{(\roman*)}]
 \item $|\chi(n)|=1$ for every $n\geq1$; and
 \item $\chi(mn)=\chi(m)\chi(n)$ for every $m,n\geq1$.
\end{enumerate}
Indeed, if $z=(z_1,z_2,\ldots)\in\bS^{\infty}$ and
$\mathfrak p=\{\mathfrak p_j\}_{j\geq1}$ is the increasing sequence of prime
numbers, then the corresponding character is determined by
\[
 \chi(2)=z_1,\qquad \chi(3)=z_2,\qquad
 \ldots,\qquad \chi(\mathfrak p_j)=z_j,\qquad \ldots,
\]
and extended multiplicatively. For a character $\chi\in\bS^{\infty}$ and $t\in\mathbb{R}$, define
$$(\mathcal{T}_t\chi)(n):=n^{-i t}\chi(n),\quad n=1,2,\ldots$$
It is known that the map $\mathcal{T}_t$ preserves the Haar measure $m_{\infty}$ on $\mathbb{T}^{\infty}$.

Given a Dirichlet series $f(s)=\sum_{n=1}^{\infty}a_nn^{-s}$ and
$\chi\in\bS^{\infty}$, its vertical limit function $f_{\chi}$ is formally defined by
\[
 f_{\chi}(s):=\sum_{n=1}^{\infty}a_n\chi(n)n^{-s}.
\]
If $f$ converges uniformly on $\overline{\mathbb C_\theta}$, then Kronecker's
theorem implies that for every $\chi\in\bS^{\infty}$ there exists a sequence
$\{\tau_k\}\subset\mathbb R$ such that $f(s+i\tau_k)$ converges to
$f_\chi(s)$ uniformly on $\overline{\mathbb C_\theta}$; see, for instance,
\cite[Lemma 2.4]{HLS}.

The Hardy space $\mathscr H^2$, introduced by Hedenmalm, Lindqvist and Seip
\cite{HLS}, is the space of Dirichlet series
$f(s)=\sum_{n=1}^{\infty}a_nn^{-s}$ satisfying
\[
 \|f\|_{\mathscr H^2}^2:=\sum_{n=1}^{\infty}|a_n|^2<\infty.
\]
For general $0<p<\infty$, the Hardy space $\mathscr H^p$ can be defined by
means of the Bohr lift. By the fundamental theorem of arithmetic, each
$n\in\mathbb N$ has a unique representation
$n=\mathfrak p^{\alpha(n)}=\mathfrak p_1^{\alpha_1(n)}\mathfrak p_2^{\alpha_2(n)}\cdots$,
where $\alpha(n)=(\alpha_1(n),\alpha_2(n),\ldots)\in\mathbb N_0^{(\infty)}$
is an eventually null sequence of non-negative integers. The Bohr lift of $f(s)=\sum_{n=1}^{\infty}a_nn^{-s}$ is the formal Fourier series
\[
 \mathcal Bf(z):=\sum_{n=1}^{\infty}a_nz^{\alpha(n)},
 \qquad z=(z_1,z_2,\ldots)\in\bS^{\infty}.
\]
For $0<p<\infty$, $\mathscr H^p$ is the completion of Dirichlet
polynomials with respect to the (quasi-)norm
\[
 \|P\|_{\mathscr H^p}
 :=\left(\int_{\bS^{\infty}}|\mathcal BP(z)|^p\,dm_{\infty}(z)\right)^{1/p}.
\]
The Bohr lift is an isometric isomorphism from $\mathscr H^p$ onto the
analytic Hardy space $H^p(\bS^{\infty})$; see
\cite{Bayart,BBSS,HLS,QueffelecQueffelec}. For $0<p<\infty$, let
$\mathscr H_0^p$ denote the closed subspace of $\mathscr H^p$ consisting of
Dirichlet series $f$ satisfying $f(+\infty)=0$.

To describe the behavior of the vertical limit functions $f_{\chi}$ for $f\in\mathscr H^p$, we introduce
the conformally invariant Hardy space. 
For $0<p<\infty$ and $\theta\in\mathbb R$, let
$\varphi_\theta(z)=\theta+(1+z)/(1-z)$. The space
$H_{\i}^p(\mathbb C_\theta)$ consists of analytic functions $h$ on
$\mathbb C_\theta$ such that $h\circ\varphi_\theta\in H^p(\mathbb D)$.
Writing $h^*$ for the boundary values on $\Ree s=\theta$, we have
\[
\|h\|_{H_{\i}^p(\mathbb C_\theta)}^p
:=\|h\circ\varphi_\theta\|_{H^p(\mathbb D)}^p
=\frac1\pi\int_{\mathbb R}|h^*(\theta+it)|^p\frac{dt}{1+t^2}.
\]
We also recall that the usual Hardy space $H^p(\mathbb C_{\theta})$ consists of
analytic functions $f$ on $\mathbb C_{\theta}$ satisfying
\[
\|f\|_{H^p(\mathbb C_{\theta})}^p
:=\sup_{\sigma>\theta}\int_{\mathbb R}|f(\sigma+it)|^p\,dt<\infty.
\]
It is well-known that for every $f\in\mathscr H^p$, the vertical limit $f_\chi$ extends analytically
to an element of $H_{\i}^p(\mathbb C_0)$ for $m_\infty$-almost every
$\chi\in\bS^\infty$; see \cite[Theorem 5]{Bayart}. Moreover, Fubini's theorem gives
$$
 \|f\|_{\mathscr H^p}^p
 =\int_{\bS^\infty}\|f_\chi\|_{H_{\i}^p(\mathbb C_0)}^p
 \,dm_\infty(\chi).
$$
We will use these almost everywhere analytic representatives of the vertical limit functions repeatedly.

For $\alpha>0$ and $\tau\in\mathbb R$, let
\[
 \Gamma_\tau(\alpha)
 :=\{\sigma+it\in\mathbb C_0:|t-\tau|<\alpha\sigma\}.
\]
We write $\Gamma_\tau:=\Gamma_\tau(1)$. Given an analytic function $f$ on
$\mathbb C_0$, its Lusin area integral is
\[
 Sf(\tau)
 :=\left(\int_{\Gamma_\tau}|f'(\sigma+it)|^2\,d\sigma\,dt\right)^{1/2},
 \qquad \tau\in\mathbb R.
\]
The Dirichlet series analogue of the Calder\'on area theorem asserts that, for
$0<p<\infty$ and $f\in\mathscr H^p$,
\begin{equation}\label{eq:Calderon-area}
 \|f\|_{\mathscr H^p}^p
 \asymp |f(+\infty)|^p
 +\int_{\bS^\infty}\int_{\mathbb R}S(f_\chi)(\tau)^p
 \frac{d\tau}{1+\tau^2}\,dm_\infty(\chi),
\end{equation}
with constants depending only on $p$; see
\cite{BrevigPerfekt,BrevigPerfektCorrection,ChenWangIntegration}. We also recall the Littlewood--Paley formula
\begin{equation}\label{eq:Littlewood-Paley}
 \|f\|_{\mathscr H^p}^p
 \asymp_p |f(+\infty)|^p+
 \int_{\bS^\infty}\int_0^\infty
 |f_\chi(\sigma)|^{p-2}|f_\chi'(\sigma)|^2
 \sigma\,d\sigma\,dm_\infty(\chi);
\end{equation}
see \cite[Theorem 5.1]{BQS}.

We now recall the area operators introduced in \cite{ChenWangYuan}. For $l>0$ and a positive Borel measure $\mu$ on $\mathbb C_0$, the area operator $\mathbb A_{\mu,l}$ is defined for $f\in\bigcup_{0<p<\infty}\mathscr{H}^p$ by
$$
 \mathbb A_{\mu,l}f(\chi,\tau)
 :=\left(\int_{\Gamma_\tau}|f_\chi(s)|^l
 \frac{d\mu(s)}{\Re(s)}\right)^{1/l},
 \qquad \chi\in\bS^\infty,\quad \tau\in\mathbb R.
$$
For $0<p<\infty$, let $L_{\i}^p(\bS^\infty\times\mathbb R)$ be the space of
measurable functions $f$ on $\bS^\infty\times\mathbb R$ such that
\[
 \|f\|_{L_{\i}^p(\bS^\infty\times\mathbb R)}^p
 :=\int_{\bS^\infty}\int_{\mathbb R}|f(\chi,\tau)|^p
 \frac{d\tau}{1+\tau^2}\,dm_\infty(\chi)<\infty.
\]
For simplicity, we write $L_{\i}^p:=L_{\i}^p(\bS^\infty\times\mathbb R)$ below. For $X=\mathscr H^p$ or $\mathscr H_0^p$, boundedness of
$\mathbb A_{\mu,l}:X\to L_{\i}^p$ means that
$\|\mathbb A_{\mu,l}f\|_{L_{\i}^p}\le C\|f\|_X$ for all $f\in X$;
its operator (quasi-)norm is the least such $C$.
A map is called compact if it is continuous and maps bounded sets into
relatively compact sets.
For a positive Borel measure $\mu$, the preceding measurable representative of
$f_\chi$ makes $\mathbb A_{\mu,l}f$ well defined for almost every
$(\chi,\tau)\in \bS^\infty\times\mathbb R$, with the value $+\infty$ allowed.

We next recall the Carleson measure conditions used in the sequel.
For $\theta\in\mathbb R$ and a finite interval
$I\subset\mathbb R$, put
\[
Q_\theta(I)
:=\{\sigma+it:\theta<\sigma\le\theta+|I|,\ t\in I\},
\]
where $|I|$ denotes the length of $I$. We write $Q(I)=Q_0(I)$ for the case $\theta=0$. A positive Borel
measure $\mu$ on $\mathbb C_{\theta}$ is called a Carleson measure if
\[
 \|\mu\|_{CM(\mathbb C_{\theta})}
 :=\sup_{0<|I|<+\infty}
 \frac{\mu\big(Q_\theta(I)\big)}{|I|}<\infty.
\]
A Carleson measure is called vanishing on $\mathbb C_\theta$ if
\[
\lim_{h\to0^+}\sup_{|I|\le h}\frac{\mu(Q_{\theta}(I))}{|I|}=0.
\]
Recall that $H_{\i}^p(\mathbb C_\theta)$ is the conformally invariant Hardy space on $\C_{\theta}$. A positive Borel measure $\mu$ on $\mathbb C_\theta$ is called an
$H_{\i}^p(\mathbb C_\theta)$-Carleson measure if the embedding
$I_\mu:H_{\i}^p(\mathbb C_\theta)\to L^p(\C_{\theta},\mu)$ is bounded, and it is called a
compact $H_{\i}^p(\mathbb C_\theta)$-Carleson measure if the same embedding is
compact. When $\theta=0$ and $I_\mu$ is bounded, we write
\[
 \|I_\mu\|^p
 :=\sup_{\|f\|_{H_{\i}^p(\mathbb C_0)}\le1}
 \int_{\mathbb C_0}|f(s)|^p\,d\mu(s).
\]

The following theorem, proved in \cite[Theorems 2.3 and 2.5]{ChenWangYuan}, establishes sufficient conditions for the area operators $\mathbb A_{\mu,l}$ to be bounded on the Hardy spaces $\mathscr{H}^p$ or $\mathscr{H}^p_0$.

\begin{theorem}[\cite{ChenWangYuan}]\label{thm:CYW}
Let $0<p,l<\infty$.
\begin{enumerate}[label=\textup{(\alph*)}]
\item If $\mu$ is a Carleson measure on $\mathbb C_0$, then
$
 \mathbb A_{\mu,l}:\mathscr H_0^p\to L_{\i}^p
$
is bounded, and
\[
 \|\mathbb A_{\mu,l}\|_{\mathscr H_0^p\to L_{\i}^p}
 \le C_{p,l}\|\mu\|_{CM(\mathbb C_0)}^{1/l}.
\]
\item If $\mu$ is an $H_{\i}^p(\mathbb C_0)$-Carleson measure, then
$
 \mathbb A_{\mu,l}:\mathscr H^p\to L_{\i}^p
$
is bounded, and
\[
 \|\mathbb A_{\mu,l}\|_{\mathscr H^p\to L_{\i}^p}
 \le C_{p,l}\|I_\mu\|^{p/l}.
\]
\end{enumerate}
\end{theorem}

We also need the horizontal translations ${V_{\delta}}$, defined for Dirichlet series $f$ by
\[
 {V_\delta} f(s):=f(s+\delta),\qquad \delta>0.
\]
Under the Bohr lift, ${V_\delta}$ corresponds to the coordinatewise dilation
$(z_j)_{j\ge1}\mapsto(\mathfrak p_j^{-\delta}z_j)_{j\ge1}$ and is contractive
on $\mathscr H^p$.  It was proved in \cite[Theorem C]{BrevigPerfektCorrection} that for any $\delta\geq0$ and $f\in\mathscr H_0^p$,
\begin{equation}\label{eq:zero-translation-decay}
 \|{V_\delta} f\|_{\mathscr H^p}\le 2^{-\delta}\|f\|_{\mathscr H^p}.
\end{equation}
The following theorem indicates that ${V_{\delta}}$ is compact on the Hardy spaces $\mathscr{H}^p$ whenever $\delta>0$.

\begin{theorem}\label{thm:translation-compact}
	For every $\delta>0$ and every $0<p<\infty$, the map
	$
	{V_\delta}:\mathscr H^p\to\mathscr H^p
	$
	is compact. The same is true on $\mathscr H_0^p$.
\end{theorem}

\begin{proof}
	Fix $0<p<\infty$ and $\delta>0$. Work first with analytic polynomials in the Bohr model and put
	\[
	D_{\boldsymbol r}f(z)=f(r_1z_1,r_2z_2,\ldots),
	\qquad r_j=\mathfrak p_j^{-\delta}.
	\]
	Thus $D_{\boldsymbol r}$ represents ${V_\delta}$. For $N\ge1$, let
	\[
	\Pi_Nf(z)=f(z_1,\ldots,z_N,0,0,\ldots).
	\]
	Successive applications of the sub-mean value inequality to the tail variables
	show that $\Pi_N$ is contractive on $H^p(\bS^\infty)$ (see \cite[Theorem 2.1]{BBSS}).
	Consequently,
	\[
	\|I-\Pi_N\|\le C_p,
	\]
	where one may take $C_p=2$ for $p\ge1$ and $C_p=2^{1/p}$ for $0<p<1$.
	
	We claim that
	\[
	\|D_{\boldsymbol r}(I-\Pi_N)\|
	\le C_pr_{N+1}.
	\]
	Indeed, let $g=(I-\Pi_N)f$, so that
	$g(z',0)=0$ for $z'=(z_1,\ldots,z_N)$. Denote by
	$R_\rho^{(N)}$ simultaneous dilation by $\rho$ in all variables after the
	$N$th one. For fixed $z'=(z_1,\ldots,z_N)\in\mathbb{T}^N$ and  $z''=(0,\ldots,0,z_{N+1},z_{N+2},\ldots)\in\mathbb{T}^{\infty}$, the function
	\[
	h_{z',z''}(\lambda)
	=\frac{g(z',\lambda z'')}{\lambda}
	\]
	has a removable singularity at the origin and is analytic in $\mathbb D$.
	Monotonicity of its one-variable Hardy means gives
	\[
	\frac1{2\pi}\int_0^{2\pi}
	|g(z',\rho e^{i\theta}z'')|^p\,d\theta
	\le \rho^p\frac1{2\pi}\int_0^{2\pi}
	|g(z',e^{i\theta}z'')|^p\,d\theta.
	\]
	Integrating in $z',z''$ and using Haar invariance under the common tail
	rotation $z''\mapsto e^{i\theta}z''$ yields
	\[
	\|R_\rho^{(N)}g\|_{H^p(\bS^{\infty})}
	\le \rho\|g\|_{H^p(\bS^{\infty})}.
	\]
	This estimate extends from analytic polynomials to all functions of
	$H^p(\bS^\infty)$ by density.
	
	For $j>N$ we have $r_j\le r_{N+1}$. Hence the tail part of
	$D_{\boldsymbol r}$ factors as coordinatewise contractions with radii
	$r_j/r_{N+1}$ followed by $R_{r_{N+1}}^{(N)}$. The dilations in the
	first $N$ coordinates are contractive as well. Therefore,
	\[
	\|D_{\boldsymbol r}(I-\Pi_N)f\|_{H^p(\bS^{\infty})}
	\le r_{N+1}\|(I-\Pi_N)f\|_{H^p(\bS^{\infty})}
	\le C_pr_{N+1}\|f\|_{H^p(\bS^{\infty})},
	\]
	and thus
	\[
	\|D_{\boldsymbol r}-D_{\boldsymbol r}\Pi_N\|
	\le C_p\mathfrak p_{N+1}^{-\delta}\longrightarrow0
	\quad \text{as}\quad N\to\infty.
	\]
	
	For fixed $N$, the map $D_{\boldsymbol r}\Pi_N$ factors through
	$H^p(\bS^N)$ and is a strict dilation in each of its finitely many
	variables. To see that it is compact, let $\{f_k\}$ be a bounded sequence in
	$H^p(\bS^\infty)$. Iterated subharmonicity shows that
	$\{\Pi_Nf_k\}$ is locally uniformly bounded on $\mathbb D^N$. By Montel's
	theorem, there exists a subsequence converging uniformly on compact subsets of
	$\mathbb D^N$, in particular on the set
	\[
	\{(r_1\zeta_1,\ldots,r_N\zeta_N):\zeta\in\bS^N\}.
	\]
	The corresponding subsequence of $\{D_{\boldsymbol r}\Pi_Nf_k\}$ therefore
	converges uniformly on $\bS^N$, and hence in $H^p(\bS^{\infty})$.
	
	It follows that $D_{\boldsymbol r}$ is a uniform limit of compact maps. Using the triangle inequality in the case $p\geq1$ and the fact that $H^p(\bS^{\infty})$ is a complete metric space under the metric
	\[
	d_p(f,g)=\|f-g\|_{H^p(\bS^{\infty})}^p
	\]
	in the case $0<p<1$, we obtain that $D_{\boldsymbol r}$ is compact on $H^p(\bS^{\infty})$. Hence ${V_\delta}$ is compact on
	$\mathscr H^p$. Since $\mathscr H_0^p$ is a closed
	${V_\delta}$-invariant subspace, the restriction is compact there as well.
\end{proof}

\begin{remark}
For the case $1\leq p<\infty$, the preceding conclusion is contained in \cite[Corollary 3]{Bayart} and \cite[Theorem 1.3]{BQS}. The proof above is included to cover the case $0<p<1$ and to provide the form used below.
\end{remark}

The following maximal estimate is a key tool in the proof of
the boundedness and compactness characterization in
Theorem~\ref{thm:general-exponents}.
For $\alpha>0$ and $n\in\{1,2\}$, define
\begin{equation}\label{eq:cone-maximal-definition}
 N_{n,\alpha}f(\chi,\tau)
 :=\sup_{\sigma+it\in\Gamma_\tau(\alpha)}
       n^\sigma|f_\chi(\sigma+it)|.
\end{equation}
By continuity in $s=\sigma+\i t$, the supremum in
\eqref{eq:cone-maximal-definition} can be taken over
$(\mathbb Q+\i\mathbb Q)\cap\Gamma_\tau(\alpha)$.
For each fixed $s$, the function $f_\chi(s)$ is measurable in
$\chi$, and the characteristic function $\mathbf1_{\Gamma_{\tau}(\alpha)}$ is measurable in
$\tau$. Thus $N_{n,\alpha}f$ is measurable in $(\chi,\tau)$ as a
countable supremum of measurable functions.

\begin{lemma}\label{lem:cone-maximal}
Let $0<p<\infty$ and $\alpha>0$. If $n=1$ and $f\in\mathscr H^p$,
or if $n=2$ and $f\in\mathscr H_0^p$, then for every $\tau\in\mathbb R$,
\begin{equation}\label{eq:cone-maximal-bound}
 \int_{\bS^\infty}N_{n,\alpha}f(\chi,\tau)^p\,dm_\infty(\chi)
 \le C_{p,\alpha}\|f\|_{\mathscr H^p}^p,
\end{equation}
where the constant does not depend on $n$ or $\tau$.
\end{lemma}

\begin{proof}
For $n=\alpha=1$, the fact that $N_{1,1}f(\chi,\tau)=N_{1,1}f(\mathcal{T}_{\tau}\chi,0)$ together with the Haar invariance and \cite[Theorem A]{BrevigPerfektCorrection} gives
\[
 \int_{\bS^\infty}N_{1,1}f(\chi,\tau)^p\,dm_\infty(\chi)
 =\int_{\mathbb R}\int_{\bS^\infty}N_{1,1}f(\chi,u)^p
     \,dm_\infty(\chi)\frac{du}{\pi(1+u^2)}
 \le C_p\|f\|_{\mathscr H^p}^p.
\]
If $n=1$ and $0<\alpha\le1$, the result follows from
$\Gamma_\tau(\alpha)\subset\Gamma_\tau(1)$.

It remains to treat $n=1$, $\alpha>1$, and $n=2$, $\alpha>0$. For
$h\in L^2(\bS^\infty)$, set
\[
 \mathcal Mh(\chi)
 :=\sup_{R>0}\frac1{2R}\int_{-R}^{R}
            |h(\mathcal{T}_u\chi)|\,du.
\]
Let $\mathcal M_A$ restrict the supremum to $0<R\le A$.
For $L>A$, put
$$
 H_\chi(u):=h(\mathcal{T}_u\chi)
                 \mathbf1_{[-L-A,L+A]}(u).
$$
Combining Fubini's theorem with the Haar invariance yields
\begin{align*}
 \int_{\bS^\infty}\|H_\chi\|_{L^2(\mathbb R)}^2\,dm_\infty(\chi)
 =\int_{-L-A}^{L+A}\int_{\bS^\infty}
       |h(\mathcal{T}_u\chi)|^2\,dm_\infty(\chi)\,du=(2L+2A)\|h\|_{L^2(\bS^\infty)}^2<\infty.
\end{align*}
Thus $H_\chi\in L^2(\mathbb R)$ for almost every $\chi\in\bS^{\infty}$.
If $|t|\le L$ and $0<R\le A$, then
$[t-R,t+R]\subset[-L-A,L+A]$. Writing
$\mathcal M_{\mathbb R}$ for the centered Hardy--Littlewood
maximal operator on $\mathbb R$, the definition of $\mathcal M_A$
and the change of variable $u=t+v$ give
\begin{align*}
 \mathcal M_Ah(\mathcal{T}_t\chi)
 &=\sup_{0<R\le A}\frac1{2R}
           \int_{-R}^{R}|h(\mathcal{T}_{t+v}\chi)|\,dv\\
 &=\sup_{0<R\le A}\frac1{2R}
           \int_{t-R}^{t+R}|H_\chi(u)|\,du
 \le\mathcal M_{\mathbb R}H_\chi(t).
\end{align*}
Hence, by the one-dimensional $L^2$ maximal
inequality, we have
\begin{align*}
 \|\mathcal M_Ah\|_{L^2(\bS^\infty)}^2
 &=\frac1{2L}\int_{\bS^\infty}\int_{-L}^{L}
       |\mathcal M_Ah(\mathcal{T}_t\chi)|^2\,dt\,dm_\infty(\chi)\\
 &\le\frac1{2L}\int_{\bS^\infty}\int_{\mathbb R}
        |\mathcal M_{\mathbb R}H_\chi(t)|^2\,dt\,dm_\infty(\chi)\\
 &\le\frac C{2L}\int_{\bS^\infty}\int_{-L-A}^{L+A}
        |h(\mathcal{T}_u\chi)|^2\,du\,dm_\infty(\chi)\\
 &=C\left(1+\frac AL\right)\|h\|_{L^2(\bS^\infty)}^2.
\end{align*}
Letting $L\to\infty$ and then $A\to\infty$, and using the monotone convergence theorem, we obtain
\begin{equation}\label{eq:flow-maximal-bound}
 \|\mathcal Mh\|_{L^2(\bS^\infty)}
 \le C\|h\|_{L^2(\bS^\infty)}.
\end{equation}

Let $f(s)=\sum_{m=1}^{N}a_m m^{-s}$ be a Dirichlet polynomial,
with $a_1=0$ when $n=2$.
Write $F=\mathcal Bf$ and $g_\chi(s)=n^sf_\chi(s)$.
Then $g_\chi$ is bounded on $\mathbb C_0$ and
$|g_\chi(i u)|=|F(\mathcal{T}_u\chi)|$. Put
$P_\sigma(v)=\sigma/[\pi(\sigma^2+v^2)]$. The sub-harmonic property gives
\[
 |g_\chi(\sigma+it)|^{p/2}
 \le\int_{\mathbb R}P_\sigma(t-u)
       |F(\mathcal{T}_u\chi)|^{p/2}\,du.
\]
If $|t-\tau|<\alpha\sigma$, then
\[
 P_\sigma(t-u)\le2(1+\alpha^2)P_\sigma(\tau-u)\quad \text{and} \quad
 P_\sigma(\tau-u)\le C\sum_{j=0}^\infty
 \frac{2^{-j}\mathbf1_{\{|u-\tau|<2^{j+1}\sigma\}}}
      {2^{j+2}\sigma}.
\]
Consequently,
\begin{align*}
 \int_{\mathbb R}P_\sigma(\tau-u)
           |F(\mathcal{T}_u\chi)|^{p/2}\,du
 &\le C\sum_{j=0}^\infty 2^{-j}
       \frac1{2^{j+2}\sigma}
       \int_{|u-\tau|<2^{j+1}\sigma}
           |F(\mathcal{T}_u\chi)|^{p/2}\,du\\
 &\le C\sum_{j=0}^\infty2^{-j}
       \mathcal M(|F|^{p/2})(\mathcal{T}_{\tau}\chi)
 \le C\mathcal M(|F|^{p/2})(\mathcal{T}_{\tau}\chi).
\end{align*}
Thus, for $\sigma+it\in\Gamma_\tau(\alpha)$, we have
\[
 |g_\chi(\sigma+it)|^{p/2}
 \le C_{\alpha}\mathcal M(|F|^{p/2})(\mathcal{T}_{\tau}\chi).
\]
Taking the supremum, applying \eqref{eq:flow-maximal-bound},
and using the Haar invariance, we obtain
\[
 \int_{\bS^\infty}N_{n,\alpha}f(\chi,\tau)^p\,dm_\infty(\chi)
 \le C_{p,\alpha}\|\mathcal M(|F|^{p/2})\|_{L^2(\bS^\infty)}^2
 \le C_{p,\alpha}\|F\|_{L^p(\bS^\infty)}^p.
\]
Since $\|F\|_{L^p(\bS^\infty)}=\|f\|_{\mathscr H^p}$,
this proves \eqref{eq:cone-maximal-bound} for polynomials.

For general $f$, polynomial density gives
$P_j(s)=\sum_{m=1}^{N_j}a_{j,m}m^{-s}$
such that $\sum_j\|P_j-f\|_{\mathscr H^p}^p<\infty$.
If $f\in\mathscr H_0^p$, these polynomials can be chosen with
$a_{j,1}=0$ for every $j$.
Fix a compact set $K\subset\mathbb C_0$ and choose a bounded open
neighborhood $E_K$ of $K$ with $\overline{E_K}\subset\mathbb C_0$.
Choose $r>0$ such that
$D(s,r)\subset E_K$ for every $s\in K$.
For almost every $\chi$, the function $|(P_j-f)_\chi|^p$
is subharmonic, so, writing $w=\sigma+it$,
\[
 |(P_j-f)_\chi(s)|^p
 \le\frac1{\pi r^2}\int_{D(s,r)}|(P_j-f)_\chi(w)|^p\,d\sigma\,dt
 \le\frac1{\pi r^2}\int_{E_K}|(P_j-f)_\chi(w)|^p\,d\sigma\,dt.
\]
Taking the supremum over $s\in K$ and  applying Fubini's theorem yields
\begin{align*}
 \int_{\bS^\infty}\sup_{s\in K}|(P_j-f)_\chi(s)|^p
                      \,dm_\infty(\chi)&\le\frac1{\pi r^2}\int_{E_K}\int_{\bS^\infty}
       |(P_j-f)_\chi(\sigma+it)|^p\,dm_\infty(\chi)\,d\sigma\,dt\\
 &=\frac1{\pi r^2}\int_{E_K}
       \|{V_\sigma}(P_j-f)\|_{\mathscr H^p}^p\,d\sigma\,dt\\
& \le\frac{|E_K|}{\pi r^2}\|P_j-f\|_{\mathscr H^p}^p,
\end{align*}
where the equality follows from the Bohr correspondence and the Haar invariance
under $\chi\mapsto\mathcal{T}_t\chi$, and the last inequality is due to that ${V_\sigma}$ is contractive on $\mathscr{H}^p$.

Summing in $j$ and using Fubini's theorem shows that
$\sum_j\sup_{s\in K}|(P_j-f)_\chi(s)|^p<\infty$ for almost
every $\chi$. Apply this conclusion to the sets
\[
 K_k=\{\sigma+\i t:1/k\le\sigma\le k,\ |t|\le k\},
 \qquad k=2,3,\ldots
\]
Since every compact subset of $\mathbb C_0$ is contained in some $K_k$,
we obtain $P_{j,\chi}\to f_\chi$ locally uniformly for almost every $\chi$, which gives
$N_{n,\alpha}f\le\liminf_jN_{n,\alpha}P_j$, and so, by Fatou's lemma,
\[
 \int_{\bS^\infty}N_{n,\alpha}f(\chi,\tau)^p\,dm_\infty(\chi)
 \le\liminf_{j\to\infty}
       \int_{\bS^\infty}N_{n,\alpha}P_j(\chi,\tau)^p\,dm_\infty(\chi)
 \le C_{p,\alpha}\|f\|_{\mathscr H^p}^p.
\]
This completes the proof.
\end{proof}

We record a continuity fact that will be useful for the nonlinear
area operators. Let $L_{\i,+}^r$ denote the cone of almost-everywhere
nonnegative elements of $L_{\i}^r$. When $0<r<1$, $L_{\i}^r$ is regarded as a complete metric space with the metric
$d_r(U,V)=\|U-V\|_{L_{\i}^r}^r$.

\begin{lemma}\label{lem:root-map}
For $0<p,l<\infty$, the map
\[
R_l(U):=U^{1/l}
\]
is continuous from $L_{\i,+}^{p/l}$ to $L_{\i,+}^{p}$.
\end{lemma}

\begin{proof}
If $l\ge1$, then $0<1/l\le1$ and
\[
 |u^{1/l}-v^{1/l}|^p\le |u-v|^{p/l},\qquad u,v\ge0,
\]
so the assertion is immediate. Let $0<l<1$ and put $a=1/l>1$. Then
\[
 |u^a-v^a|\le a|u-v|(u^{a-1}+v^{a-1}),\qquad u,v\ge0.
\]
Taking the $p$th power and applying H\"older's inequality with
exponents $a$ and $a/(a-1)$, we obtain
\[
 \|U^a-V^a\|_{L_{\i}^p}
 \le C_{p,l}\|U-V\|_{L_{\i}^{pa}}
 \big(\|U\|_{L_{\i}^{pa}}^{a-1}+\|V\|_{L_{\i}^{pa}}^{a-1}\big).
\]
Since $pa=p/l$, the desired continuity also follows.
\end{proof}

We shall also use the following standard approximation principle for compact maps.

\begin{lemma}\label{lem:uniform-compact-approx}
Let $\Phi:X\to Y$ be a map from a quasi-normed space $X$ to a complete metric space $(Y,d)$. Suppose that there is a sequence of compact maps $\Phi_j:X\to Y$ such that for every $R>0$,
\[
\lim_{j\to\infty}\sup_{\|x\|_X\le R} d(\Phi_j x,\Phi x)=0.
\]
Then $\Phi:X\to Y$ is compact.
\end{lemma}

\begin{proof}
	Let $B_R=\{x:\|x\|_X\le R\}$. Fix $\eta>0$. Choose $j$ such that
	$d(\Phi_jx,\Phi x)<\eta/3$ on $B_R$. Since $\Phi_j(B_R)$ is relatively compact,
	it is totally bounded and has a finite $\eta/3$-net. The same net, enlarged by the
	$\eta/3$ approximation, gives a finite $\eta$-net for $\Phi(B_R)$. Thus
	$\Phi(B_R)$ is totally bounded. The completeness of $Y$ gives the desired compactness.
\end{proof}

\section{A counterexample}\label{sec:counterexample}

The following conjecture concerning necessary conditions for the area operators $\mathbb{A}_{\mu,l}$ to be bounded on Hardy spaces of Dirichlet series was proposed in \cite{ChenWangYuan}.

\begin{conjecture}[{\cite[Conjecture 5.1]{ChenWangYuan}}]
\label{conj:CYW}
Let $0<p,l<\infty$, and let $\mu$ be a positive Borel measure on
$\mathbb C_0$.
\begin{enumerate}[label=\textup{(\alph*)}]
\item If
$\mathbb{A}_{\mu,l}:\mathscr H^p_0\to L_{\i}^p$ is bounded, then $\mu$
is a Carleson measure on $\mathbb C_{1/2}$.
\item If
$\mathbb{A}_{\mu,l}:\mathscr H^p\to L_{\i}^p$ is bounded, then $\mu$ is
an $H_{\i}^p(\mathbb C_{1/2})$-Carleson measure.
\end{enumerate}
\end{conjecture}

The main purpose of this section is to disprove the above conjecture by constructing a counterexample. To this end, we need the following lemma. Here and in the sequel, we write $\delta_s$ to denote the Dirac point mass at $s\in\mathbb{C}$.

\begin{lemma}\label{lem:separated}
Let $0<p,l<\infty$ and suppose that $s_k=\sigma_k+it_k\in\mathbb{C}_0$ such that the intervals
\[
 I_k=(t_k-\sigma_k,t_k+\sigma_k)
\]
are pairwise disjoint.  Put
\[
 \mu=\sum_{k\ge1}m_k\delta_{s_k},\qquad
 w_k=\int_{I_k}\frac{d\tau}{1+\tau^2},
 \qquad m_k>0.
\]
Then, for every $f\in\mathscr H^p$,
\begin{equation}\label{eq:separated-formula}
 \norm{\mathbb{A}_{\mu,l}f}_{L_{\i}^p}^p
 =\sum_{k\ge1}
   \left(\frac{m_k}{\sigma_k}\right)^{p/l}
   w_k\,\norm{{V_{\sigma_k}}f}_{\mathscr H^p}^p,
\end{equation}
where both sides are allowed to be infinite.  Consequently,
$\mathbb{A}_{\mu,l}:\mathscr H^p\to L_{\i}^p$ is bounded if and only if
\begin{equation}\label{eq:separated-sum}
 S_{p,l}(\mu):=
 \sum_{k\ge1}
 \left(\frac{m_k}{\sigma_k}\right)^{p/l}w_k<\infty,
\end{equation}
and in this case
\[
 \norm{\mathbb{A}_{\mu,l}}_{\mathscr H^p\to L_{\i}^p}^p
 =S_{p,l}(\mu).
\]
The condition \eqref{eq:separated-sum} is sufficient for boundedness on
$\mathscr H^p_0$ and is also necessary when
$\sup_k\sigma_k<\infty$.
\end{lemma}

\begin{proof}
Fix $f\in\mathscr{H}^p$. Since $s_k=\sigma_k+it_k$, we have
\[
 s_k\in\Gamma_\tau
 \quad\Longleftrightarrow\quad
 |t_k-\tau|<\sigma_k
 \quad\Longleftrightarrow\quad
 \tau\in I_k.
\]
The intervals $I_k$ are pairwise disjoint, so for each fixed $\tau$ at most one atom $s_k$ lies in $\Gamma_\tau$. 
Combining this fact with Fubini's theorem yields
\begin{align*}
\int_{\bS^{\infty}}\int_{\mathbb{R}}
  \mathbb{A}_{\mu,l}f(\chi,\tau)^p
  \frac{d\tau}{1+\tau^2}dm_{\infty}(\chi)
&=\sum_{k\ge1}\left(\frac{m_k}{\sigma_k}\right)^{p/l}w_k
  \int_{\bS^\infty}|f_\chi(\sigma_k+it_k)|^p\,dm_\infty(\chi)\\
&=\sum_{k\ge1}\left(\frac{m_k}{\sigma_k}\right)^{p/l}w_k
  \,\norm{{V_{\sigma_k}}f}_{\mathscr H^p}^p,
\end{align*}
where the last equality follows from the rotation invariance of the measure $m_{\infty}$. This establishes \eqref{eq:separated-formula}.

Since each ${V_{\sigma_k}}$ is contractive on $\mathscr{H}^p$,
\eqref{eq:separated-sum} implies the boundedness of $\mathbb{A}_{\mu,l}$.  Conversely, testing
\eqref{eq:separated-formula} with $f\equiv1$ gives necessity and the
exact norm on $\mathscr H^p$.  On $\mathscr H^p_0$, sufficiency is
unchanged.  If $\sup_{k\geq1}\sigma_k<\infty$, then testing with
$f(s)=2^{-s}$ gives
\[
 \norm{{V_{\sigma_k}}f}_{\mathscr H^p}^p
 =2^{-p\sigma_k}\ge2^{-p\sup_{k\geq1}\sigma_k},
\]
and hence necessity follows.
\end{proof}

We are now ready to construct an example that indicates both parts of Conjecture \ref{conj:CYW} are false for every
$0<p,l<\infty$, even when $\mu$ is a probability measure.

\begin{theorem}\label{thm:counterexample}
There is a probability measure $\mu$ on $\mathbb C_0$, which is neither an ordinary Carleson
measure nor an $H_{\i}^p$-Carleson measure on either
$\mathbb C_0$ or $\mathbb C_{1/2}$, such that for every $0<p,l<\infty$,
\[
 \mathbb{A}_{\mu,l}:\mathscr H^p\to L_{\i}^p
 \quad\text{and}\quad
 \mathbb{A}_{\mu,l}:\mathscr H^p_0\to L_{\i}^p
\]
are both bounded.
\end{theorem}

\begin{proof}
For $k\ge2$, set
\[
\begin{array}{lll}
 \varepsilon_k=2^{-3k},&
 \sigma_k=\frac12+\varepsilon_k,&
 t_k=2^k,\\[2mm]
 \widetilde\sigma_k=2^{-3k},&
 \widetilde t_k=-2^{k^2},&
 m_k=\widetilde m_k=2^{-k}.
\end{array}
\]
Define
$$
 \mu=\sum_{k=2}^\infty m_k\delta_{\sigma_k+it_k}
     +\sum_{k=2}^\infty
       \widetilde m_k\delta_{\widetilde\sigma_k+i\widetilde t_k}.
$$
The total mass of $\mu$ is
$2\sum_{k=2}^\infty2^{-k}=1$.  The intervals centered at $t_k$ and
$\widetilde t_k$ with respective radii $\sigma_k$ and
$\widetilde\sigma_k$ are pairwise disjoint.

For the first family,
\[
 w_k:=\int_{t_k-\sigma_k}^{t_k+\sigma_k}
          \frac{d\tau}{1+\tau^2}
 \le C\frac{\sigma_k}{1+t_k^2}\le C2^{-2k}.
\]
Since $\sigma_k\asymp1$,
\[
 \sum_{k\geq2}\left(\frac{m_k}{\sigma_k}\right)^{p/l}w_k
 \le C_{p,l}\sum_{k\geq2}2^{-kp/l-2k}<\infty.
\]
For the second family,
\[
 \widetilde w_k
 \le C\frac{\widetilde\sigma_k}{1+\widetilde t_k^2}
 \le C2^{-3k-2k^2},
\]
and hence
\[
 \sum_{k\geq2}\left(\frac{\widetilde m_k}{\widetilde\sigma_k}\right)^{p/l}
 \widetilde w_k
 \le C\sum_{k\geq2}2^{-2k^2+(2p/l-3)k}<\infty.
\]
Lemma \ref{lem:separated} therefore gives the asserted boundedness on both spaces.

We next verify failure of the Carleson measure conditions.  Let $J_k$
be the interval centered at $t_k$ of length $\varepsilon_k$.  Then
$\sigma_k+it_k\in Q_{1/2}(J_k)$ and
\[
 \frac{\mu(Q_{1/2}(J_k))}{|J_k|}
 \ge\frac{m_k}{\varepsilon_k}=2^{2k}\longrightarrow\infty.
\]
Thus $\mu$ is not a Carleson measure on $\mathbb C_{1/2}$.  Similarly,
if $\widetilde J_k$ is centered at $\widetilde t_k$ and has length
$\widetilde\sigma_k$, then
\[
 \frac{\mu(Q_0(\widetilde J_k))}{|\widetilde J_k|}
 \ge\frac{\widetilde m_k}{\widetilde\sigma_k}
 =2^{2k}\longrightarrow\infty,
\]
so $\mu$ is not a Carleson measure on $\mathbb C_0$.

Finally, the half-plane Carleson characterization used in
\cite[Lemma 3.3]{ChenWangYuan} says that if $\mu$ is an
$H_{\i}^p(\mathbb C_\theta)$-Carleson measure, then
\[
 d\nu_\theta(s)
 =\bigl(1+(\Ree s-\theta)^2+(\Imm s)^2\bigr)\,d\mu(s)
\]
is a  Carleson measure on $\mathbb C_\theta$.  For
$\theta=1/2$, on the boxes $Q_{1/2}(J_k)$ we have
\[
 \frac{\nu_{1/2}(Q_{1/2}(J_k))}{|J_k|}
 \ge t_k^2\frac{m_k}{\varepsilon_k}
 =2^{4k}\longrightarrow\infty.
\]
For $\theta=0$, on the boxes $Q_0(\widetilde J_k)$ we have
\[
 \frac{\nu_0(Q_0(\widetilde J_k))}{|\widetilde J_k|}
 \ge \widetilde t_k^{\,2}
     \frac{\widetilde m_k}{\widetilde\sigma_k}
 =2^{2k^2+2k}\longrightarrow\infty.
\]
The two $H_{\i}^p$-Carleson measure conditions therefore fail as well.
\end{proof}

\section{An integrability characterization of boundedness and compactness}
\label{sec:diagonal}

In this section, we characterize the boundedness and compactness of $\mathbb A_{\mu,l}$ on $\mathscr H^p$ and $\mathscr H_0^p$ in terms of integrability conditions on $\mu$.  For $0<p,l<\infty$ and a positive Borel measure $\mu$ on
$\mathbb C_0$, put $X_1=\mathscr H^p$ and $X_2=\mathscr H_0^p$,
and define, for $n\in\{1,2\}$,
\begin{equation*}
 B_n(\tau)
 :=\int_{\Gamma_\tau}n^{-l\Ree s}\frac{d\mu(s)}{\Ree s},
 \qquad
 J_n(\mu;p,l)
 :=\int_{\mathbb R}B_n(\tau)^{p/l}
       \frac{d\tau}{1+\tau^2}.
\end{equation*}
Both integrals are allowed to take the value $+\infty$.
Fatou's lemma shows that
$B_n$ is lower semicontinuous in $\tau$, hence measurable.

\begin{theorem}\label{thm:general-exponents}
Let $0<p,l<\infty$ and let $\mu$ be a positive Borel measure
on $\mathbb C_0$. For each $n\in\{1,2\}$, the following statements
are equivalent:
\begin{enumerate}[label=\textup{(\roman*)}]
\item $\mathbb A_{\mu,l}:X_n\to L_{\i}^p$ is bounded;
\item $\mathbb A_{\mu,l}:X_n\to L_{\i}^p$ is compact;
\item $J_n(\mu;p,l)<\infty$.
\end{enumerate}
When these conditions hold,
\begin{equation}\label{eq:general-norm}
 J_n(\mu;p,l)
 \le\norm{\mathbb{A}_{\mu,l}}_{X_n\to L_{\i}^p}^{p}
 \le C_p J_n(\mu;p,l).
\end{equation}
Here $C_p$ is independent of $\mu$, $l$ and $n$. If $p\ge l$, then
\begin{equation}\label{eq:general-exact-norm}
 \norm{\mathbb{A}_{\mu,l}}_{X_n\to L_{\i}^p}^{p}
 =J_n(\mu;p,l).
\end{equation}
\end{theorem}

\begin{proof}
Fix $n\in\{1,2\}$. Suppose first that $J_n(\mu;p,l)<\infty$.
By \eqref{eq:cone-maximal-definition}, for any $f\in X_n$ and
$s\in\Gamma_\tau$,
$|f_\chi(s)|\le n^{-\Ree s}N_{n,1}f(\chi,\tau)$.
Combining this estimate with Lemma \ref{lem:cone-maximal}
and Fubini's theorem yields
\begin{align}
 \norm{\mathbb{A}_{\mu,l}f}_{L_{\i}^p}^p
 &=\int_{\mathbb R}\int_{\bS^\infty}
    \left(\int_{\Gamma_\tau}|f_\chi(s)|^l
          \frac{d\mu(s)}{\Ree s}\right)^{p/l}
          \,dm_\infty(\chi)\frac{d\tau}{1+\tau^2}\notag\\
 &\le\int_{\mathbb R}B_n(\tau)^{p/l}
       \int_{\bS^\infty}N_{n,1}f(\chi,\tau)^p
          \,dm_\infty(\chi)\frac{d\tau}{1+\tau^2}\notag\\
 &\le C_pJ_n(\mu;p,l)\norm{f}_{\mathscr H^p}^p.
 \label{eq:general-upper}
\end{align}
Hence $\mathbb{A}_{\mu,l}:X_n\to L_{\i}^p$ is bounded.
Conversely, if $\mathbb{A}_{\mu,l}$ is bounded on $X_n$, then
testing with $f_n(s)=n^{-s}$, for which
$\norm{f_n}_{\mathscr H^p}=1$ and
$|(f_n)_\chi(s)|=n^{-\Ree s}$, gives
\begin{equation}\label{eq:general-test}
 \mathbb{A}_{\mu,l}f_n(\chi,\tau)^l=B_n(\tau)
 \quad \text{and} \quad
 \norm{\mathbb{A}_{\mu,l}}_{X_n\to L^p_{\i}}^p\geq
 \norm{\mathbb{A}_{\mu,l}f_n}_{L_{\i}^p}^p
 =J_n(\mu;p,l).
\end{equation}
This proves the necessity and, together with
\eqref{eq:general-upper}, the norm estimates in
\eqref{eq:general-norm}.

When $p\ge l$, we obtain the exact quasi-norm formula as follows.
By the rotation invariance of the measure $m_\infty$, for any
$f\in X_n$ and $s=\sigma+it\in\mathbb C_0$, we have
\[
 \int_{\bS^\infty}|f_\chi(\sigma+it)|^p\,dm_\infty(\chi)
 =\norm{{V_\sigma} f}_{\mathscr H^p}^p.
\]
Therefore, Minkowski's inequality in $L^{p/l}(\bS^\infty)$ gives
\begin{align*}
 \left(\int_{\bS^\infty}
       \mathbb{A}_{\mu,l}f(\chi,\tau)^p
          \,dm_\infty(\chi)\right)^{l/p}
 &\le\int_{\Gamma_\tau}
       \left(\int_{\bS^\infty}|f_\chi(s)|^p
             \,dm_\infty(\chi)\right)^{l/p}
          \frac{d\mu(s)}{\Ree s}\\
 &=\int_{\Gamma_\tau}
       \norm{{V_{\Ree s}}f}_{\mathscr H^p}^{\,l}
          \frac{d\mu(s)}{\Ree s}\\
 &\le B_n(\tau)\norm{f}_{\mathscr H^p}^{\,l},
\end{align*}
where the last inequality follows from the contractivity of ${V_\sigma}$
when $n=1$, and from \eqref{eq:zero-translation-decay} when $n=2$.
Integrating in $\tau$ then gives
\eqref{eq:general-exact-norm}.

We next prove the compactness of $\mathbb{A}_{\mu,l}$ when
$J_n(\mu;p,l)<\infty$. For a positive measure $\nu\le\mu$, write
\[
 \Phi_\nu(f):=(\mathbb{A}_{\nu,l}f)^l,\qquad f\in X_n.
\]
Thus $\Phi_\nu(f)$ is nonnegative and
\[
 \norm{\Phi_\nu(f)}_{L_{\i}^{p/l}}
 =\norm{\mathbb{A}_{\nu,l}f}_{L_{\i}^p}^{\,l}.
\]
For $\varepsilon>0$, write $\mu=\mu_\varepsilon+\mu^\varepsilon$
as in Section \ref{sec:preliminaries}. Since $B_n(\tau)<\infty$
for almost every $\tau$, the dominated convergence theorem gives
$$
 J_n(\mu_\varepsilon;p,l)\to0
 \quad\text{as}\quad\varepsilon\to0.
$$
Hence, for every $R>0$, \eqref{eq:general-upper} gives
\begin{align}
 \sup_{\norm{f}_{X_n}\le R}
 \norm{\Phi_\mu(f)-\Phi_{\mu^\varepsilon}(f)}_{L_{\i}^{p/l}}
 &=\sup_{\norm{f}_{X_n}\le R}
       \norm{\Phi_{\mu_\varepsilon}(f)}_{L_{\i}^{p/l}}\notag\\
 &=\sup_{\norm{f}_{X_n}\le R}
       \norm{\mathbb{A}_{\mu_\varepsilon,l}f}_{L_{\i}^p}^{\,l}
       \notag\\
 &\le C_p^{l/p}R^lJ_n(\mu_\varepsilon;p,l)^{l/p}\to0
       \quad\text{as}\quad\varepsilon\to0.
 \label{eq:general-boundary-approx}
\end{align}

Fix $\varepsilon>0$. We claim that
$\Phi_{\mu^\varepsilon}:X_n\to L_{\i}^{p/l}$ is compact.
If $s=\sigma+it\in\Gamma_\tau$ and $\sigma\ge\varepsilon$, then
$s-\varepsilon/2\in\Gamma_\tau(2)$, since
$|t-\tau|<\sigma\le2(\sigma-\varepsilon/2)$. Thus, for any $h\in X_n$, we have
\[
 |h_\chi(s)|
 =|({V_{\varepsilon/2}}h)_\chi(s-\varepsilon/2)|
 \le n^{-\sigma+\varepsilon/2}
       N_{n,2}({V_{\varepsilon/2}}h)(\chi,\tau).
\]
Since ${V_{\varepsilon/2}}h\in X_n$, this estimate together with Lemma
\ref{lem:cone-maximal} gives
\begin{align}
 \norm{\mathbb{A}_{\mu^\varepsilon,l}h}_{L_{\i}^p}^p
 &\le n^{p\varepsilon/2}
       \int_{\mathbb R}B_n(\tau)^{p/l}
       \int_{\bS^\infty}
          N_{n,2}({V_{\varepsilon/2}}h)(\chi,\tau)^p
          \,dm_\infty(\chi)\frac{d\tau}{1+\tau^2}\notag\\
 &\le C_p n^{p\varepsilon/2}J_n(\mu;p,l)
       \norm{{V_{\varepsilon/2}}h}_{\mathscr H^p}^p.
 \label{eq:general-interior-estimate}
\end{align}
Let $\{f_j\}$ be a bounded sequence in $X_n$. By Theorem
\ref{thm:translation-compact}, after passing to a subsequence if
necessary, we may assume that $\{{V_{\varepsilon/2}}f_j\}$ converges
in $X_n$. If $0<l\le1$, then
\[
 \big||a|^l-|b|^l\big|\le|a-b|^l,\qquad a,b\in\mathbb C,
\]
and therefore, using \eqref{eq:general-interior-estimate}, we obtain
$$\begin{aligned}
 \norm{\Phi_{\mu^\varepsilon}(f_j)
       -\Phi_{\mu^\varepsilon}(f_k)}_{L_{\i}^{p/l}}
 &\le\norm{\mathbb{A}_{\mu^\varepsilon,l}(f_j-f_k)}_{L_{\i}^p}^{\,l}
       \notag\\
 &\le C_p^{l/p}n^{l\varepsilon/2}J_n(\mu;p,l)^{l/p}
       \norm{{V_{\varepsilon/2}}(f_j-f_k)}_{\mathscr H^p}^{\,l}.
\end{aligned}$$
If $l>1$, the inequality
\[
 \big||a|^l-|b|^l\big|
 \le l|a-b|(|a|+|b|)^{l-1},\qquad a,b\in\mathbb C,
\]
together with H\"older's inequality in the cone integral gives a
pointwise estimate. Raising this estimate to the power $p/l$ and applying H\"older's
inequality in $(\chi,\tau)$ with exponents $l$ and $l/(l-1)$, we obtain
$$\begin{aligned}
 \norm{\Phi_{\mu^\varepsilon}(f_j)
       -\Phi_{\mu^\varepsilon}(f_k)}_{L_{\i}^{p/l}}
 &\le C_{p,l}
       \norm{\mathbb{A}_{\mu^\varepsilon,l}(f_j-f_k)}_{L_{\i}^p}
       \left(\norm{\mathbb{A}_{\mu^\varepsilon,l}f_j}_{L_{\i}^p}
             +\norm{\mathbb{A}_{\mu^\varepsilon,l}f_k}_{L_{\i}^p}
       \right)^{l-1}\notag\\
 &\le C_{p,l}n^{l\varepsilon/2}J_n(\mu;p,l)^{l/p}
       \norm{{V_{\varepsilon/2}}(f_j-f_k)}_{\mathscr H^p}\notag\\
 &\qquad\times
       \left(\norm{{V_{\varepsilon/2}}f_j}_{\mathscr H^p}
             +\norm{{V_{\varepsilon/2}}f_k}_{\mathscr H^p}
       \right)^{l-1}.
\end{aligned}$$
Thus $\{\Phi_{\mu^\varepsilon}(f_j)\}$ converges in $L_{\i}^{p/l}$, and the claim is proved.
Consequently, by \eqref{eq:general-boundary-approx} and Lemma
\ref{lem:uniform-compact-approx}, $\Phi_\mu$ is compact from $X_n$
into $L_{\i,+}^{p/l}$. Lemma \ref{lem:root-map} shows that
\[
 R_l:L_{\i,+}^{p/l}\to L_{\i,+}^p,
 \qquad R_l(U)=U^{1/l},
\]
is continuous. Since $\mathbb{A}_{\mu,l}=R_l\circ\Phi_\mu$,
the desired compactness follows.

Conversely, compactness obviously implies boundedness. The proof is complete.
\end{proof}

\begin{remark}
Although the area operator $\mathbb{A}_{\mu,l}$ is nonlinear, we have that its boundedness and continuity from $\mathscr{H}^p$ (or $\mathscr{H}^p_0$) to $L^p_{\i}$ are equivalent for all $0<p,l<\infty$.
In fact, since
\[
 \mathbb{A}_{\mu,l}(\lambda f)=|\lambda|\mathbb{A}_{\mu,l}f,\quad \lambda\in\mathbb{C},\quad f\in\mathscr{H}^p,
\]
the continuity of $\mathbb{A}_{\mu,l}$ at the origin implies its boundedness.  Conversely, suppose $\mathbb{A}_{\mu,l}$ is bounded, i.e.
\[
 \norm{\mathbb{A}_{\mu,l}f}_{L_{\i}^p}\le C\norm f_{\mathscr{H}^p},\quad f\in\mathscr{H}^p.
\]
If $l\ge1$, Minkowski's inequality gives
\[
 |\mathbb{A}_{\mu,l}f-\mathbb{A}_{\mu,l}g|\le \mathbb{A}_{\mu,l}(f-g),
\]
and hence $\mathbb{A}_{\mu,l}$ is Lipschitz continuous.  If $0<l<1$, set
\[
 \Phi_{\mu,l}(f)=(\mathbb{A}_{\mu,l}f)^l.
\]
The scalar inequality
\[
 \big||a|^l-|b|^l\big|\le|a-b|^l,\quad a,b\in\mathbb{C}
\]
gives
\[
 \norm{\Phi_{\mu,l}(f)-\Phi_{\mu,l}(g)}_{L_{\i}^{p/l}}
 \le C^l\norm{f-g}_{\mathscr H^p}^l.
\]
Lemma \ref{lem:root-map} then yields the continuity of
$\mathbb{A}_{\mu,l}=R_l\circ\Phi_{\mu,l}$. Thus the boundedness of $\mathbb{A}_{\mu,l}$ is equivalent to its continuity. Theorem \ref{thm:general-exponents} actually shows that boundedness, continuity and compactness are equivalent for all $0<p,l<\infty$.
\end{remark}

We now apply Theorem \ref{thm:general-exponents} to the measure
constructed in Theorem \ref{thm:counterexample}.

\begin{corollary}\label{cor:diagonal-counterexample-compact}
For the probability measure $\mu$ in Theorem \ref{thm:counterexample}
and every $0<p,l<\infty$, both
\[
 \mathbb{A}_{\mu,l}:\mathscr H^p\to L_{\i}^p
 \quad\text{and}\quad
 \mathbb{A}_{\mu,l}:\mathscr H^p_0\to L_{\i}^p
\]
are compact.  Nevertheless, $\mu$ is neither an ordinary Carleson
measure nor an $H_{\i}^p$-Carleson measure on either
$\mathbb C_0$ or $\mathbb C_{1/2}$.
\end{corollary}

For $p\le l$, Theorem \ref{thm:general-exponents} shows that
finiteness of the measure alone ensures compactness.

\begin{corollary}\label{cor:finite-measure-general}
If $0<p\le l<\infty$ and $\mu(\mathbb C_0)<\infty$, then
$\mathbb A_{\mu,l}$ is compact from each $X_n$ into $L_{\i}^p$, and
\[
 \|\mathbb A_{\mu,l}\|_{X_n\to L_{\i}^p}^{p}
 \le C_p\pi^{1-p/l}\bigl(2\mu(\mathbb C_0)\bigr)^{p/l},
 \qquad n=1,2.
\]
\end{corollary}
\begin{proof}
Since $B_2\le B_1$ and $p/l\le1$,
Fubini's theorem together with H\"{o}lder's inequality gives
\[
 J_n(\mu;p,l)\le J_1(\mu;p,l)
 \le \pi^{1-p/l}
       \left(\int_{\mathbb R}B_1(\tau)
                   \frac{d\tau}{1+\tau^2}\right)^{p/l}
 \le \pi^{1-p/l}\bigl(2\mu(\mathbb C_0)\bigr)^{p/l}.
\]
The conclusion follows from Theorem \ref{thm:general-exponents}.
\end{proof}

For $p>l$, Theorem \ref{thm:general-exponents} also yields
the following necessary local vanishing box estimate.
\begin{proposition}\label{prop:general-local-box}
Suppose $p>l>0$ and $\mathbb A_{\mu,l}:X_n\to L_{\i}^p$
is bounded. For each $T>0$ there is a constant $C_{T,p,l}$
such that, for every $I\subset[-T,T]$ with $|I|\le1$,
\[
 \mu(Q(I))
 \le C_{T,p,l}\|\mathbb A_{\mu,l}\|_{X_n\to L_{\i}^p}^{\,l}
                |I|^{1-l/p}.
\]
Moreover,
\begin{equation}\label{eq:general-box-vanishing}
 \lim_{h\rightarrow0}
 \sup_{\substack{I\subset[-T,T]\\0<|I|\le h}}
 \frac{\mu(Q(I))}{|I|^{1-l/p}}=0.
\end{equation}
\end{proposition}
\begin{proof}
For $s=\sigma+it\in Q(I)$, we have
$(t-\sigma,t+\sigma)\subset3I$ and $\sigma\le1$.
Consequently,
\[
 \int_{3I}B_n(\tau)\,d\tau
 \ge2\int_{Q(I)}n^{-l\sigma}\,d\mu(s)
 \ge2n^{-l}\mu(Q(I)).
\]
Then, H\"older's inequality with exponents $p/l$ and
$p/(p-l)$, together with the fact that 
$d\tau \asymp d\tau/(1+\tau^2)$ on $[-T-1,T+1]$, yields
\[
 \mu(Q(I))\le C_{T,p,l}|I|^{1-l/p}
       \left(\int_{3I}B_n(\tau)^{p/l}
                   \frac{d\tau}{1+\tau^2}\right)^{l/p}.
\]
This, combined with \eqref{eq:general-test}, gives the first assertion.
Since $B_n^{p/l}$ is integrable and
$\int_{3I}(1+\tau^2)^{-1}d\tau\le3|I|$,
absolute continuity of the integral gives
\eqref{eq:general-box-vanishing}.
\end{proof}

The following example applies Theorem \ref{thm:general-exponents}
to measures supported on the real axis and gives the exact parameter
ranges for boundedness and compactness.
For $p>l$, they also show that the vanishing box condition
\eqref{eq:general-box-vanishing} is not sufficient for boundedness,
even when it holds uniformly over all intervals in $\mathbb R$.

\begin{example}
\begin{enumerate}[label=(\alph*),leftmargin=*]
\item For $0<\alpha<1$, let $\mu_\alpha$ be the probability
measure on the real segment $(0,1)$ given by
\[
 \int_{\mathbb C_0}H(s)\,d\mu_\alpha(s)
 =\alpha\int_0^1H(\sigma)\sigma^{\alpha-1}\,d\sigma.
\]
For $I_h=(-h/2,h/2)$ and $0<h<1$, we obtain
\[
 \frac{\mu_\alpha(Q(I_h))}{|I_h|}
 =\frac{\alpha}{h}\int_0^h\sigma^{\alpha-1}\,d\sigma
 =h^{\alpha-1}\rightarrow\infty
 \quad \text{as} \quad h\rightarrow0,
\]
so $\mu_\alpha$ is not a Carleson measure.
On the real axis, $\sigma\in\Gamma_\tau$ if and only if
$|\tau|<\sigma$. Thus, for $0<|\tau|<1$ and $n=1,2$, we obtain
\[
 n^{-l}B_1(\tau)\le B_n(\tau)
 =\alpha\int_{|\tau|}^1n^{-l\sigma}\sigma^{\alpha-2}\,d\sigma
 \le B_1(\tau).
\]
Moreover,
\[
 B_1(\tau)
 =\alpha\int_{|\tau|}^1\sigma^{\alpha-2}\,d\sigma
 =\frac{\alpha}{1-\alpha}
      \bigl(|\tau|^{\alpha-1}-1\bigr).
\]
Hence $\lim_{\tau\to0}B_1(\tau)|\tau|^{1-\alpha}
=\alpha/(1-\alpha)$, while $B_n(\tau)=0$ for $|\tau|\ge1$.
Since $1/2\le(1+\tau^2)^{-1}\le1$ for $|\tau|\le1$, we have
\[
 \begin{aligned}
 J_n(\mu_\alpha;p,l)<\infty
 &\iff\int_0^{1/2}\tau^{(\alpha-1)p/l}\,d\tau<\infty\\
 &\iff(\alpha-1)p/l>-1
 \iff\alpha>1-l/p.
 \end{aligned}
\]
By Theorem \ref{thm:general-exponents}, the corresponding
operator $\mathbb{A}_{\mu_{\alpha},l}$ on both spaces is bounded and compact precisely when
$\alpha>1-l/p$. Choosing
$\max\{0,1-l/p\}<\alpha<1$ therefore gives a compact
example with bounded imaginary support and a non-Carleson measure.
For $p>l$, the condition $\mu(\mathbb C_0)<\infty$
does not ensure boundedness, since
\[
 0<\alpha\le1-l/p
 \quad\Longrightarrow\quad
 \mu_\alpha(\mathbb C_0)=1,\qquad
 J_1(\mu_\alpha;p,l)=J_2(\mu_\alpha;p,l)=\infty.
\] Thus the range $p\le l$ in Corollary~\ref{cor:finite-measure-general} is sharp. This example also shows that the exponent $1-l/p$ in
Proposition~\ref{prop:general-local-box} is sharp.
Indeed, for fixed $p>l$ and any $\gamma>1-l/p$, choose
$1-l/p<\alpha<\min\{\gamma,1\}$.
The corresponding area operator on both spaces is compact,
but
\[
\frac{\mu_\alpha(Q(I_h))}{|I_h|^\gamma}
=h^{\alpha-\gamma}\rightarrow\infty
\quad\text{as } \quad h\rightarrow0.
\]
Thus $1-l/p$ cannot be replaced by any larger exponent,
even in the local box estimate.

\item 
Let $p>l$ and put $\alpha=1-l/p$.
At this critical exponent, the area operator $\mathbb{A}_{\mu_{\alpha},l}$ is unbounded by \textup{(a)}, and the measure $\mu_\alpha$ does not satisfy
\eqref{eq:general-box-vanishing}, since
$\mu_\alpha(Q(I_h))/|I_h|^\alpha=1$ for $0<h<1$.
We now modify the measure so that it satisfies this box
condition uniformly over all intervals in $\mathbb R$,
while the associated area operator remains unbounded. Let the measure $\mu$ on $\C_0$ be defined by
\[
\int_{\mathbb C_0}H(s)\,d\mu(s)
=\int_0^{e^{-1}}H(\sigma)
\frac{\sigma^{\alpha-1}}{(\log(1/\sigma))^{l/p}}\,d\sigma.
\]
Since $\log(1/\sigma)\ge1$ on $(0,e^{-1})$,
$\mu(\mathbb C_0)\le\int_0^{e^{-1}}\sigma^{\alpha-1}\,d\sigma
=e^{-\alpha}/\alpha<\infty$.
For any interval $I$ with $0<|I|\le h<e^{-1}$, we have
\[
 \mu(Q(I))
 \le\frac{1}{(\log(1/|I|))^{l/p}}
       \int_0^{|I|}\sigma^{\alpha-1}\,d\sigma
 =\frac{|I|^\alpha}{\alpha(\log(1/|I|))^{l/p}}.
\]
Hence
\[
 0\le\sup_{0<|I|\le h}\frac{\mu(Q(I))}{|I|^\alpha}
 \le\frac{1}{\alpha(\log(1/h))^{l/p}}
 \rightarrow0, \quad h\rightarrow0.
\]
For $0<|\tau|<1/(2e)$ and $n=1,2$, we have
\[
 B_n(\tau)
 \ge n^{-l}\int_{|\tau|}^{2|\tau|}
       \frac{\sigma^{\alpha-2}}{(\log(1/\sigma))^{l/p}}\,d\sigma
 \ge n^{-l}2^{\alpha-2}
       \frac{|\tau|^{\alpha-1}}{(\log(1/|\tau|))^{l/p}}.
\]
Using the fact $(\alpha-1)p/l=-1$ and $(1+\tau^2)^{-1}\ge1/2$ on
$(0,1/(2e))$ gives
\[
 \begin{aligned}
 J_n(\mu;p,l)
 \ge n^{-p}2^{(\alpha-2)p/l-1}
       \int_0^{1/(2e)}\frac{d\tau}{\tau\log(1/\tau)}=\infty.
 \end{aligned}
\]
Thus $\mu$ satisfies even the global vanishing box condition,
but $J_1(\mu;p,l)=J_2(\mu;p,l)=\infty$.
Theorem \ref{thm:general-exponents} shows that neither corresponding
operator is bounded. Therefore the local vanishing condition in
Proposition \ref{prop:general-local-box} is not sufficient for
boundedness.
\end{enumerate}
\end{example}

Combining the boundedness result in Theorem \ref{thm:CYW}(b)
with the equivalence of boundedness and compactness in
Theorem \ref{thm:general-exponents}, we obtain the following
corollary.

\begin{corollary}\label{cor:Hip-carleson-compact}
Let $0<p,l<\infty$. If $\mu$ is an
$H_{\i}^p(\mathbb C_0)$-Carleson measure, then
\[
 \mathbb{A}_{\mu,l}:\mathscr H^p\to L_{\i}^p
\]
is compact; consequently its restriction to $\mathscr H^p_0$ is also
compact.
\end{corollary}

\begin{remark}
Corollary \ref{cor:Hip-carleson-compact} requires only boundedness, not
compactness, of the $H_{\i}^p$ embedding.  In contrast,
Corollary \ref{cor:diagonal-counterexample-compact} shows that the Carleson measure condition is not necessary even for the compactness of area operators $\mathbb{A}_{\mu,l}$ in general.
\end{remark}

\section{Direct proofs of compactness under Carleson conditions}
\label{sec:compactness}

In this section, we establish two compactness results under some Carleson measure conditions. Theorem \ref{thm:general-exponents}, together with Theorem
\ref{thm:CYW}, already implies these results. We give direct
proofs to obtain estimates that can also be used for measures
depending on $\chi$.

The two approaches share a compactness framework: both work
with $\Phi_\mu(f)=(\mathbb A_{\mu,l}f)^l$ and then use the
continuous root map in Lemma \ref{lem:root-map}, with the
interior part treated using the compactness of ${V_\delta}$. However,
the required estimates are obtained differently: in Section
\ref{sec:diagonal}, the dominated convergence theorem gives
$J_n(\mu_\varepsilon;p,l)\to0$, and Lemma
\ref{lem:cone-maximal} controls the area integral after the
$\chi$-independent cone weight has been taken outside the Haar
integral. In the proofs below, the boundary or tail contribution
is controlled directly by vanishing Carleson norms or compact
embedding estimates, and the interior part by a translated
embedding.

To apply this method to measures depending on $\chi$, we need
estimates for each $f_\chi$ with constants independent of $\chi$.
We also need the boundary or tail contributions to tend to zero
uniformly in $\chi$. We establish the required uniform estimates
and apply this method to Volterra operators in Section
\ref{sec:volterra}.
We begin by estimating the Carleson norm of the boundary part
$\mu_\varepsilon$.

\begin{lemma}\label{lem:vanishing-small}
Let $\mu$ be a vanishing Carleson measure on $\mathbb C_0$. Then
\[
 \norm{\mu_\varepsilon}_{CM(\mathbb{C}_0)}\to0
 \quad \text{as}\quad \varepsilon\to0^+.
\]
\end{lemma}

\begin{proof}
Set
\[
 \eta(\varepsilon)
 =\sup_{|I|\le\varepsilon}\frac{\mu(Q(I))}{|I|}.
\]
Then $\eta(\varepsilon)\to0$.  Let $I$ be a finite interval.  If
$|I|\le\varepsilon$, then
\[
 \mu_\varepsilon(Q(I))\le\eta(\varepsilon)|I|.
\]
If $|I|>\varepsilon$, cover $I$ by intervals $J_j$ of length exactly
$\varepsilon$, allowing the two end intervals to extend beyond $I$.
They can be chosen so that
\[
 \sum_j|J_j|\le3|I|.
\]
Because $\mu_\varepsilon$ is supported in $0<\sigma<\varepsilon$,
\[
 Q(I)\cap\{0<\sigma<\varepsilon\}
 \subset\bigcup_jQ(J_j).
\]
It follows that
\[
 \mu_\varepsilon(Q(I))
 \le\sum_j\mu(Q(J_j))
 \le3\eta(\varepsilon)|I|.
\]
Dividing both sides by $|I|$ and taking the supremum over $I$ completes the proof.
\end{proof}

The next lemma contains the compactness argument for the part of a measure
supported away from the boundary.

\begin{lemma}\label{lem:interior-compact}
Let $0<p,l<\infty$, $X$ denote either $\mathscr H^p$ or
$\mathscr H^p_0$, and let $\lambda$ be a positive Borel measure supported in
$\{\sigma\ge 2\delta\}$ for some $\delta\in(0,1)$. Define the translated measure $\nu$ by
\[
 \nu(E):=\lambda(E+\delta),
\]
where $E+\delta:=\{w+\delta:w\in E\}$ for every Borel set $E\subset\mathbb{C}_0$.
If
$
 \mathbb A_{\nu,l}:X\to L_{\i}^p
$
is bounded, then the map
$
 \Phi_\lambda(f):=\bigl(\mathbb A_{\lambda,l}f\bigr)^l
$
is compact from $X$ into $L_{\i,+}^{p/l}$. Consequently,
$\mathbb A_{\lambda,l}:X\to L_{\i}^p$ is compact.
\end{lemma}

\begin{proof}
Note that $\operatorname{supp}\nu\subset\{\Re w\ge\delta\}$. For $g\in X$ define
\[
 B_\delta g(\chi,\tau)^l
 :=\int_{\{w:w+\delta\in\Gamma_\tau\}}
 |g_\chi(w)|^l\frac{d\nu(w)}{\Re w+\delta}.
\]
If $s=w+\delta$, then $({V_\delta} f)_\chi(w)=f_\chi(s)$, and hence
\begin{equation}\label{eq:interior-factorization}
 \mathbb A_{\lambda,l}f=B_\delta({V_\delta} f).
\end{equation}

We first show that $B_\delta:X\to L_{\i}^p$ is bounded. Write $w=u+it$ with $u\geq\delta$.
If $w+\delta\in\Gamma_\tau$, then $|t-\tau|<u+\delta$. Since $u\ge\delta$,
\[
 \mathbf 1_{\{w+\delta\in\Gamma_\tau\}}
 \le \sum_{\eta\in\{-\delta,0,\delta\}}
 \mathbf 1_{\{w\in\Gamma_{\tau+\eta}\}}.
\]
Moreover, $(u+\delta)^{-1}\le u^{-1}$. Therefore,
\[
 B_\delta g(\chi,\tau)^l
 \le \sum_{\eta\in\{-\delta,0,\delta\}}
 \mathbb A_{\nu,l}g(\chi,\tau+\eta)^l.
\]
Note that $\delta\in(0,1)$ implies that for any $\tau\in\mathbb{R}$ and any $\eta\in\{-\delta,0,\delta\}$,
$$1+\tau^2\leq3(1+(\tau-\eta)^2).$$
Consequently, by the boundedness of $\mathbb{A}_{\nu,l}$, we have for any $g\in X$,
\begin{align*}
\norm{B_{\delta}g}_{L^p_{\i}}^p
&\leq C_{p,l}\sum_{\eta\in\{-\delta,0,\delta\}}
    \int_{\bS^{\infty}}\int_{\mathbb{R}}
    \mathbb{A}_{\nu,l}g(\chi,\tau+\eta)^p
    \frac{d\tau}{1+\tau^2}dm_{\infty}(\chi)\\
&=C_{p,l}\sum_{\eta\in\{-\delta,0,\delta\}}
    \int_{\bS^{\infty}}\int_{\mathbb{R}}
    \mathbb{A}_{\nu,l}g(\chi,\tau)^p
    \frac{d\tau}{1+(\tau-\eta)^2}dm_{\infty}(\chi)\\
&\leq3C_{p,l}\sum_{\eta\in\{-\delta,0,\delta\}}
    \int_{\bS^{\infty}}\int_{\mathbb{R}}
    \mathbb{A}_{\nu,l}g(\chi,\tau)^p
    \frac{d\tau}{1+\tau^2}dm_{\infty}(\chi)\\
&\leq 9C_{p,l}\norm{\mathbb{A}_{\nu,l}}^p\norm{g}^p_{X}.
\end{align*}
That is, $B_{\delta}:X\to L^p_{\i}$ is bounded.

Set
\[
 \Psi_\delta(g):=B_\delta(g)^l.
\]
We claim that $\Psi_\delta:X\to L_{\i}^{p/l}$ is continuous on bounded
sets. If $0<l\le1$, then
\[
 |\Psi_\delta(f)-\Psi_\delta(g)|\le \Psi_\delta(f-g),
\]
and hence
\[
 \|\Psi_\delta(f)-\Psi_\delta(g)\|_{L_{\i}^{p/l}}
 \le \|B_\delta(f-g)\|_{L_{\i}^p}^l.
\]
If $l>1$, H\"older's inequality inside the defining integral, followed by
H\"older's inequality in $\bS^\infty\times\mathbb R$, gives
\[
 \|\Psi_\delta(f)-\Psi_\delta(g)\|_{L_{\i}^{p/l}}
 \le C_{p,l}\|B_\delta(f-g)\|_{L_{\i}^p}
 \bigl(\|B_\delta f\|_{L_{\i}^p}+\|B_\delta g\|_{L_{\i}^p}\bigr)^{l-1}.
\]
Thus the boundedness of $B_\delta$ proves the claim.

By Theorem \ref{thm:translation-compact}, ${V_\delta}:X\to X$ is compact.
From \eqref{eq:interior-factorization},
$
 \Phi_\lambda=\Psi_\delta\circ {V_\delta},
$
so $\Phi_\lambda$ is compact into $L_{\i,+}^{p/l}$.
Finally, Lemma \ref{lem:root-map} gives compactness of
$\mathbb A_{\lambda,l}=R_l\circ\Phi_\lambda$.
\end{proof}

\subsection{The compactness of $\mathbb{A}_{\mu,l}$ from $\mathscr H^p_0$ to $L_{\i}^p(\bS^\infty \times \mathbb R)$}

We now establish a sufficient condition for $\mathbb{A}_{\mu,l}:\mathscr{H}^p_0\to L^p_{\i}$ to be compact.

\begin{theorem}\label{thm:H0-compact}
Let $0<p,l<\infty$.  If $\mu$ is a vanishing Carleson measure on $\mathbb C_0$, then
\[
 \mathbb{A}_{\mu,l}:\mathscr H^p_0\to L_{\i}^p(\bS^\infty \times \mathbb R)
\]
is compact.
\end{theorem}

\begin{proof}
For $\varepsilon\in(0,1)$, decompose $\mu=\mu_\varepsilon+\mu^\varepsilon$.
Throughout this proof, we write
\[
 \Phi_\lambda(f):=\bigl(\mathbb{A}_{\lambda,l}f\bigr)^l
\]
for a positive Borel measure $\lambda$.
By Lemma \ref{lem:vanishing-small} and Theorem \ref{thm:CYW}(a), we have 
\[
 \|\mathbb{A}_{\mu_\varepsilon,l}\|_{\mathscr H^p_0\to L_{\i}^p}
 \le C_{p,l}\|\mu_\varepsilon\|_{CM}^{1/l}\to0 \quad \text{as}
 \quad \varepsilon\to0^+.
\]

For the interior part, define
\[
 \nu^\varepsilon(E):=\mu^\varepsilon(E+\varepsilon/2)
\]
for each Borel set $E\subset \mathbb{C}_0$. We verify that $\nu^\varepsilon$ is a Carleson measure on $\mathbb C_0$.
Let $I\subset\mathbb{R}$ be a finite interval. If $|I|<\varepsilon/2$, then
$\nu^\varepsilon(Q(I))=0$. If $|I|\ge\varepsilon/2$, then let $J$ be the interval with
the same center as $I$ such that $|J|=2|I|$. Since $Q(I)+\varepsilon/2\subset Q(J)$,
\[
 \nu^\varepsilon(Q(I))
 \le \mu(Q(J))
 \le 2\|\mu\|_{CM}|I|.
\]
Thus $\nu^\varepsilon$ is a Carleson measure. By Theorem \ref{thm:CYW}(a),
$\mathbb A_{\nu^\varepsilon,l}:\mathscr H^p_0\to L_{\i}^p$ is bounded,
and Lemma \ref{lem:interior-compact}, applied with
$\lambda=\mu^\varepsilon$ and $\delta=\varepsilon/2$, shows that
$\Phi_{\mu^\varepsilon}$ is compact.

Finally, by the equality
\[
 \Phi_\mu(f)=\Phi_{\mu^\varepsilon}(f)+\Phi_{\mu_\varepsilon}(f),
\]
we have for every $R>0$,
\[
\begin{aligned}
 &\sup_{\norm{f}_{\mathscr H^p_0}\le R}
 \norm{\Phi_\mu(f)-\Phi_{\mu^\varepsilon}(f)}_{L_{\i}^{p/l}}\\
 &\quad=
 \sup_{\norm{f}_{\mathscr H^p_0}\le R}
 \norm{\Phi_{\mu_\varepsilon}(f)}_{L_{\i}^{p/l}}
 =
 \sup_{\norm{f}_{\mathscr H^p_0}\le R}
 \norm{\mathbb{A}_{\mu_\varepsilon,l}f}_{L_{\i}^p}^{l}
 = R^l\norm{\mathbb{A}_{\mu_\varepsilon,l}}_{\mathscr H^p_0\to L_{\i}^p}^l\to0
\end{aligned}
\]
as $\varepsilon\to0^+$. Combining this with Lemma
\ref{lem:uniform-compact-approx} shows that $\Phi_\mu$ is compact from $\mathscr{H}^p_0$ into
$L_{\i,+}^{p/l}$.  The compactness of
$\mathbb{A}_{\mu,l}=R_l\circ\Phi_\mu:\mathscr{H}^p_0\to L^p_{\i}$ then follows from Lemma \ref{lem:root-map}.
\end{proof}

\subsection{The compactness of $\mathbb{A}_{\mu,l}$ from $\mathscr H^p$ to $L_{\i}^p(\bS^\infty \times \mathbb R)$}

To establish a sufficient condition for $\mathbb{A}_{\mu,l}$ to be compact on the full space $\mathscr{H}^p$,  we need the following tail estimate for compact embeddings.

\begin{lemma}\label{lem:tail-compact-embedding}
	Let $0<p<\infty$ and let $\mu$ be a positive Borel measure on $\mathbb{C}_0$ such that the embedding
	$I_\mu:H_{\i}^p(\mathbb C_0)\to L^p(\mathbb{C}_0,\mu)$
	is compact. For each positive integer $N$, define
	\[
	K_N=\left\{\sigma+it:\frac1N\le\sigma\le N,\ |t|\le N\right\}
	\]
	and $\rho_N=\mu|_{\mathbb C_0\setminus K_N}$.  Then
	\[
	\|I_{\rho_N}\|_{H_{\i}^p(\mathbb C_0)\to L^p(\mathbb{C}_0,\rho_N)}\to0 \quad \text{as}\quad N\to\infty.
	\]
\end{lemma}

\begin{proof}
	Write $L^p(\mu):=L^p(\mathbb{C}_0,\mu)$ for simplicity. Let $B$ be the unit ball of $H_{\i}^p(\mathbb C_0)$ and set
	\[
	\mathcal K=\overline{I_\mu(B)}^{\,L^p(\mu)}.
	\]
	The set $\mathcal K$ is compact. Equip $L^p(\mu)$ with the complete metric
	\[
	d_p(U,V)=\norm{U-V}_{L^p(\mu)}^{\theta_p},
	\qquad \theta_p=\min\{1,p\}.
	\]
	For each $N$, define the operator $M_N$ by
	\[
	M_N U := \mathbf{1}_{\mathbb C_0 \setminus K_N} U, \qquad U \in L^p(\mu).
	\]
	
	First, observe that for every fixed $U \in L^p(\mu)$,
	\[
	\lim_{N\to\infty} M_N U = 0
	\]
	pointwise since $\bigcup_{N} K_N = \mathbb C_0$ and the sets $K_N$ increase to $\mathbb C_0$. Since $|M_N U| \le |U| \in L^p(\mu)$, the dominated convergence theorem yields
	\[
	\| M_N U \|_{L^p(\mu)} \to 0 \quad \text{as}\quad N \to \infty,
	\]
	and consequently
	\begin{align}\label{100}
		d_p(M_N U, 0) = \| M_N U \|_{L^p(\mu)}^{\theta_p} \to 0 \quad\text{as}\quad N \to \infty. 
	\end{align}
	
	Moreover, each $M_N$ is a contraction on $(L^p(\mu), d_p)$: indeed, for $U,V \in L^p(\mu)$,
	\[
	\| M_N U - M_N V \|_{L^p(\mu)}
	= \| \mathbf{1}_{\mathbb C_0 \setminus K_N} (U - V) \|_{L^p(\mu)}
	\le \| U - V \|_{L^p(\mu)},
	\]
	so, since $t \mapsto t^{\theta_p}$ is increasing on $[0,\infty)$,
	\begin{align}\label{99}
		d_p(M_N U, M_N V) \le d_p(U,V). 
	\end{align}
	
	We now show that the convergence in \eqref{100} is uniform on the compact set $\mathcal K$. Let $\varepsilon > 0$ be arbitrary. Since $\mathcal K$ is compact in the metric $d_p$, there exist finitely many points $U_1, \dots, U_m \in \mathcal K$ such that
	\[
	\mathcal K \subset \bigcup_{j=1}^m \left\{ U : d_p(U, U_j) < \varepsilon \right\}.
	\]
	For each fixed $j$, \eqref{100} gives some $N_j$ such that for all $N \ge N_j$,
	\[
	d_p(M_N U_j, 0) < \varepsilon.
	\]
	Let $N_0 := \max\{N_1, \dots, N_m\}$. Then for every $N \ge N_0$ and every $j=1,\dots,m$,
	\[
	d_p(M_N U_j, 0) < \varepsilon. 
	\]
	
	Now take any $U \in \mathcal K$. Choose $j$ such that $d_p(U, U_j) < \varepsilon$. Using the contraction property \eqref{99} and the triangle inequality for the metric $d_p$, we get for $N\geq N_0$,
	\[
	\begin{aligned}
		d_p(M_N U, 0)
		&\le d_p(M_N U, M_N U_j) + d_p(M_N U_j, 0) \\
		&\le d_p(U, U_j) + d_p(M_N U_j, 0) \\
		&< 2\varepsilon.
	\end{aligned}
	\]
	Thus, for all $N \ge N_0$,
	\[
	\sup_{U \in \mathcal K} d_p(M_N U, 0) \le 2\varepsilon.
	\]
	That is,
	\[
	\sup_{U \in \mathcal K} d_p(M_N U, 0) \rightarrow 0 \quad\text{as}\quad N \to \infty. 
	\]
	Hence
	\[
	\norm{I_{\rho_N}}^p= \sup_{F\in B}\int_{\mathbb C_0\setminus K_N}|F(s)|^p\,d\mu(s)\to0 \quad\text{as}\quad N \to \infty,
	\]
	which is exactly what we want.
\end{proof}

We are now ready to establish a sufficient condition for $\mathbb{A}_{\mu,l}:\mathscr{H}^p\to L^p_{\i}$ to be compact.

\begin{theorem}\label{thm:full-compact}
	Let $0<p,l<\infty$.  If
	$
	I_\mu:H_{\i}^p(\mathbb C_0)\to L^p(\mathbb{C}_0,\mu)
	$
	is compact, then
	\[
	\mathbb{A}_{\mu,l}:\mathscr H^p\to L_{\i}^p(\bS^\infty \times \mathbb R)
	\]
	is compact.
\end{theorem}

\begin{proof}
	Suppose that $I_\mu:H_{\i}^p(\mathbb C_0)\to L^p(\mathbb{C}_0,\mu)$ is compact. Then it is clear that $\mu(\mathbb{C}_0)<\infty$. For a positive measure $\lambda$ and a Dirichlet series $f$, write
	\[
	\Phi_\lambda(f):=\bigl(\mathbb{A}_{\lambda,l}f\bigr)^l.
	\]
	Let $K_N$ and $\rho_N$ be as in Lemma \ref{lem:tail-compact-embedding}, and put
	\[
	\mu_N=\mu|_{K_N}.
	\]
	By Lemma \ref{lem:tail-compact-embedding} and Theorem \ref{thm:CYW}(b),
	\[
	\|\mathbb{A}_{\rho_N,l}\|_{\mathscr H^p\to L_{\i}^p}
	\le C_{p,l}\norm{I_{\rho_N}}_{H^p_{\i}(\mathbb{C}_0)\to L^p(\rho_N)}^{p/l}
	\to0\quad\text{as}\quad N\to\infty.
	\]
	
	It remains to treat the compact part $\mu_N$. Define
	\[
	\nu_N(E):=\mu_N\bigl(E+1/(2N)\bigr).
	\]
	Then $\operatorname{supp}\nu_N$ is contained in the compact set
	\[
	K_N-1/(2N)\subset
	\left\{\frac1{2N}\le\Re w\le N-\frac1{2N},\ |\Im w|\le N\right\},
	\]
	and $\nu_N$ is finite. We claim that $\nu_N$ is an
	$H_{\i}^p(\mathbb C_0)$-Carleson measure. Choose $r=1/(4N)$. If
	$w=u+iv\in\operatorname{supp}\nu_N$, then $D(w,r)\subset\mathbb C_0$ and
	$|v|\le N$. For any $f\in H^p_{\i}(\mathbb{C}_0)$, the sub-harmonic property of $|f|^p$ gives, writing $\zeta=\sigma+it$,
	\[
	|f(w)|^p\le \frac1{\pi r^2}\int_{D(w,r)}|f(\zeta)|^p\,d\sigma\,dt
	\le C_N\|f\|_{H_{\i}^p(\mathbb C_0)}^p.
	\]
	Hence
	\[
	\int_{\mathbb C_0}|f|^p\,d\nu_N
	\le C_N\nu_N(\mathbb C_0)\|f\|_{H_{\i}^p(\mathbb C_0)}^p.
	\]
	Consequently, $\nu_N$ is an $H^p_{\i}(\mathbb{C}_0)$-Carleson measure and Theorem \ref{thm:CYW}(b) implies that
	$\mathbb A_{\nu_N,l}:\mathscr H^p\to L_{\i}^p$ is bounded. Since
	$\mu_N$ is supported in $\{\sigma\ge 1/N\}$,
	Lemma \ref{lem:interior-compact}, applied with $\lambda=\mu_N$ and
	$\delta=1/(2N)$, shows that
	$\Phi_{\mu_N}$ is compact from $\mathscr{H}^p$ to $L^{p/l}_{\i}$.
	
	By the equality
	\[
	\Phi_\mu(f)=\Phi_{\mu_N}(f)+\Phi_{\rho_N}(f),
	\]
	we obtain that for every $R>0$,
	\[
	\begin{aligned}
		&\sup_{\norm{f}_{\mathscr H^p}\le R}
		\norm{\Phi_\mu(f)-\Phi_{\mu_N}(f)}_{L_{\i}^{p/l}}\\
		&\quad=
		\sup_{\norm{f}_{\mathscr H^p}\le R}
		\norm{\Phi_{\rho_N}(f)}_{L_{\i}^{p/l}}
		=
		\sup_{\norm{f}_{\mathscr H^p}\le R}
		\norm{\mathbb{A}_{\rho_N,l}f}_{L_{\i}^p}^{l}
		= R^l\norm{\mathbb{A}_{\rho_N,l}}_{\mathscr H^p\to L_{\i}^p}^l\to0
	\end{aligned}
	\]
	as $N\to\infty$. Therefore,  Lemma \ref{lem:uniform-compact-approx} shows that $\Phi_\mu:\mathscr{H}^p\to L^{p/l}_{\i}$ is compact, and
	the compactness of $\mathbb{A}_{\mu,l}=R_l\circ\Phi_\mu:\mathscr{H}^p\to L^p_{\i}$ follows from Lemma \ref{lem:root-map}.
\end{proof}

\begin{remark}
	In view of Corollary \ref{cor:diagonal-counterexample-compact}, the conditions in Theorems \ref{thm:H0-compact} and \ref{thm:full-compact} are sufficient but not necessary.
\end{remark}

\section{Essential norms and compactness of Volterra operators}\label{sec:volterra}
In this section, we establish an estimate for the essential norm of bounded Volterra operators on \(\mathscr H^p\). As an application, we obtain compactness criteria and recover the known compactness result for symbols in $\operatorname{VMOA}(\mathbb{C}_0)$. Let $\mathcal D$ denote the class of Dirichlet series which converge in some half-plane.
 For
$g\in\mathcal D$, define for Dirichlet series $f$ by
\[
T_gf(s):=-\int_s^{+\infty}f(w)g'(w)\,dw.
\]
Brevig, Perfekt and Seip \cite{BPSVolterra} initiated the investigation of $T_g$ on $\mathscr H^p$. It was proved in \cite[Theorem 5.3]{BPSVolterra} that if $g$ belongs to $\operatorname{BMOA}(\mathbb{C}_0)$, the space of functions $h\in H_{\i}^2(\mathbb{C}_0)$ with
$$\|h\|_{\operatorname{BMO}(\mathbb{C}_0)}:
=\sup_{I\subset\mathbb{R}}\frac{1}{|I|}\int_{I}\left|h(it)
-\frac{1}{|I|}\int_{I}h(i\tau)d\tau\right|dt<\infty,$$
then $T_g$ is bounded on the Hardy space $\mathscr{H}^p$ for $0<p<\infty$. A different proof was also given in \cite[Theorem 4.1]{ChenWangYuan}. Moreover, it was pointed out in \cite[Section 7]{BPSVolterra} that the above boundedness result has its compact analogue when the
$\operatorname{BMOA}$-condition is replaced by $\operatorname{VMOA}$-condition. Recall that $\operatorname{VMOA}(\mathbb C_0)$ consists of functions $h\in\operatorname{BMOA}(\mathbb C_0)$ satisfying
$$\lim_{a\rightarrow0}\sup_{0<|I|\le a}\frac{1}{|I|}\int_{I}\left|h(it)
-\frac{1}{|I|}\int_{I}h(i\tau)d\tau\right|dt=0.$$

For $g\in\mathcal D\cap\operatorname{VMOA}(\mathbb C_0)$,
the area theorem associates $T_g$ with area integrals having
exponent $l=2$ and the measure family
\[
d\mu_{g_\chi}(\sigma+it)
:=|g_\chi'(\sigma+it)|^2\sigma\,d\sigma\,dt,
\quad \text{for a.e. }\chi\in\bS^\infty.
\]
The vertical limits exist almost everywhere, since
$g-g(+\infty)=T_g1\in\mathscr H^p$.
Unlike the fixed-measure setting of Theorem
\ref{thm:general-exponents}, here the measure depends on the same
character $\chi$ as $f_\chi$. Indeed, the maximal estimate leads to
\[
\int_{\bS^\infty}B_n(\chi,\tau)^{p/2}
N_{n,1}f(\chi,\tau)^p\,dm_\infty(\chi),
\]
where
\[
B_n(\chi,\tau)
:=\int_{\Gamma_\tau}n^{-2\Ree s}
\frac{d\mu_{g_\chi}(s)}{\Ree s},
\qquad n=1,2.
\]
Since Lemma \ref{lem:cone-maximal} controls only the unweighted Haar
integral, Theorem \ref{thm:general-exponents} cannot be applied directly.
We instead use the direct method of Section \ref{sec:compactness} with
estimates uniform in $\chi$, together with the Carleson embedding
characterization of Brevig--Perfekt--Seip \cite{BPSVolterra}. To recover the
$\operatorname{VMOA}(\mathbb C_0)$ compactness result, we need uniform
Carleson bounds and uniform boundary vanishing for the family
$\{\mu_{g_\chi}\}$. The following lemma provides these estimates.

\begin{lemma}\label{lem:uniform-vmoa-vertical}
Let $g\in\mathcal D\cap\operatorname{VMOA}(\mathbb C_0)$.
Then the measures $\mu_{g_\chi}$ satisfy
\[
 \sup_{\chi\in\bS^\infty}\|\mu_{g_\chi}\|_{CM(\mathbb{C}_0)}<\infty
\]
and
\begin{equation}\label{eq:uniform-vmoa}
 \lim_{\varepsilon\to0^+}
 \sup_{\chi\in\bS^\infty}
 \left\|\mu_{g_\chi}
 \big|_{\{0<\sigma<\varepsilon\}}\right\|_{CM(\mathbb{C}_0)}=0.
\end{equation}
\end{lemma}

\begin{proof}
By \cite[Lemma~2.1]{BPSVolterra}, vertical limits preserve the
$\operatorname{BMOA}$ norm. Hence the Carleson measure characterization of $\operatorname{BMOA}(\mathbb{C}_0)$ gives
\[
 \sup_{\chi\in\bS^\infty}\|\mu_{g_\chi}\|_{CM(\mathbb{C}_0)}<\infty;
\]
see \cite[Theorem VI.3.4]{Garnett}.

To prove \eqref{eq:uniform-vmoa}, let $\eta>0$. Since Dirichlet polynomials are dense in
$\mathcal D\cap\operatorname{VMOA}(\mathbb C_0)$ with respect to the $\operatorname{BMOA}$ norm (see
\cite[Theorem 7.1]{BPSVolterra}), we may choose a Dirichlet polynomial $P$ such that
\[
 \|g-P\|_{\operatorname{BMOA}(\mathbb C_0)}<\eta.
\]
Again by \cite[Lemma~2.1]{BPSVolterra},
\[
 \|(g-P)_\chi\|_{\operatorname{BMOA}(\mathbb C_0)}<\eta
\]
for all $\chi\in\bS^\infty$. Thus the measures associated with $(g-P)_\chi$ have Carleson norms bounded
by $C\eta^2$, uniformly in $\chi$. On the other hand, since $P$ is a
Dirichlet polynomial,
\[
C_P:=\sup_{\chi\in\bS^\infty}
\sup_{s\in\mathbb C_0}|P_\chi'(s)|<\infty.
\]
Therefore, for every finite interval $I$, we have
\[
 \begin{aligned}
 \frac{\mu_{P_\chi}(Q(I)\cap\{0<\sigma<\varepsilon\})}{|I|}
 &=\frac1{|I|}\int_I\int_0^{\min\{|I|,\varepsilon\}}
       |P_\chi'(\sigma+it)|^2\sigma\,d\sigma\,dt\\
 &\le\frac{C_P^2}{2}\min\{|I|,\varepsilon\}^2
 \le\frac{C_P^2}{2}\varepsilon^2,
 \end{aligned}
\]
uniformly in $\chi$. Since
$|g_\chi'|^2\le2|(g-P)_\chi'|^2+2|P_\chi'|^2$, we obtain
\[
 \sup_{\chi\in\bS^\infty}
 \left\|\mu_{g_\chi}
   \big|_{\{0<\sigma<\varepsilon\}}\right\|_{CM(\mathbb{C}_0)}
 \le C\eta^2+C_P^2\varepsilon^2.
\]
Since $\eta>0$ is arbitrary, the proof is complete.
\end{proof}

We now turn to essential norms and compactness. The boundedness
criterion of Brevig--Perfekt--Seip \cite[Theorem~5.1]{BPSVolterra} says that
$T_g$ is bounded on $\mathscr H^p$ if and only if the Carleson
embedding
\begin{equation}\label{eq:bps-embedding}
	\int_{\bS^\infty}\int_0^\infty
	|f_\chi(\sigma)|^p|g_\chi'(\sigma)|^2\sigma\,d\sigma\,dm_\infty(\chi)
	\le C(g,p)^2\|f\|_{\mathscr H^p}^p
\end{equation}
holds, and the smallest $C(g,p)$ is comparable to $\|T_g\|$. They also
observed in \cite[Section~7.1]{BPSVolterra} that compactness is obtained by
replacing bounded embeddings by vanishing embeddings. We combine this embedding description with boundary area integrals to obtain a quantitative essential norm formula {for bounded Volterra operators}.

{

For later use, we show that boundedness is preserved when the symbol is
translated horizontally. Suppose that $T_g$ is bounded on $\mathscr H^p$
and let $\delta>0$. By \cite[Theorem~5.3(b)]{BPSVolterra},
$g\in\mathscr H^2$, and then $(V_{\sigma}g)'\in\mathscr{H}^2$ for any $\sigma>0$ (see for instance \cite[Proposition 2.3]{Chen}). Consequently, for any $\sigma>0$ and almost every $\chi\in\bS^{\infty}$, the Poisson kernel representation gives
\[
 g_\chi'(\sigma+\delta)
 =\int_{\mathbb R}P_\delta(t)g_\chi'(\sigma+i t)\,dt,
\]
where
$$P_\delta(t)=\frac{\delta}{\pi(\delta^2+t^2)}$$
is the Poisson kernel. Then the Cauchy--Schwarz inequality gives
\[
 |g_\chi'(\sigma+\delta)|^2
 \le\int_{\mathbb R}P_\delta(t)|g_\chi'(\sigma+i t)|^2\,dt.
\]
Consequently, Fubini's theorem, Haar invariance under
$\chi\mapsto\mathcal{T}_t\chi$, and \eqref{eq:bps-embedding} applied to
$f(\cdot-i t)$ yield
\begin{align*}
 &\int_{\bS^\infty}\int_0^\infty
 |f_\chi(\sigma)|^p|({V_\delta} g)_\chi'(\sigma)|^2
 \sigma\,d\sigma\,dm_\infty(\chi)\\
 &\quad\le\int_{\mathbb R}P_\delta(t)
 \int_{\bS^\infty}\int_0^\infty
 |f_\chi(\sigma-i t)|^p|g_\chi'(\sigma)|^2
 \sigma\,d\sigma\,dm_\infty(\chi)\,dt\\
 &\quad\le C(g,p)^2\|f\|_{\mathscr H^p}^p,
\end{align*}
because vertical translations are isometries on $\mathscr H^p$ and
$\int_{\mathbb R}P_\delta(t)\,dt=1$. Thus,
\begin{equation}\label{eq:translated-volterra-bound}
 C({V_\delta} g,p)\le C(g,p),
 \quad \text{and} \quad
 \|T_{{V_\delta} g}\|\lesssim_p\|T_g\|.
\end{equation}}

We now introduce two boundary quantities to describe the essential norm and compactness of \(T_g\). For $0<\varepsilon\le1$, define
\begin{equation*}
	\mathfrak C_{g,p}(\varepsilon)
	:=\sup_{\|f\|_{\mathscr H^p}\le1}
	\left(
	\int_{\bS^\infty}\int_0^\varepsilon
	|f_\chi(\sigma)|^p|g_\chi'(\sigma)|^2\sigma\,d\sigma\,dm_\infty(\chi)
	\right)^{1/2}
\end{equation*}
and
\[
b_{g,p}(\varepsilon)
:=\sup_{\|f\|_{\mathscr H^p}\le1}
\left[
\int_{\bS^\infty}\int_{\mathbb R}
\left(
\int_{\Gamma_\tau\cap\{0<\sigma<\varepsilon\}}
|f_\chi(s)|^2|g_\chi'(s)|^2\,d\sigma\,dt
\right)^{p/2}
\frac{d\tau}{1+\tau^2}\,dm_\infty(\chi)
\right]^{1/p}.
\]
When $T_g$ is bounded, both quantities are finite and nondecreasing in $\varepsilon$, so their limits as $\varepsilon\rightarrow0$ exist.
The essential norm of $T_g$ is defined by
\[
\|T_g\|_{\mathrm e}
:=\inf\left\{\|T_g-K\|_{\mathscr H^p\to\mathscr H^p}:
K:\mathscr H^p\to\mathscr H^p\ \text{ is compact and linear}\right\}.
\]
With this notation, we obtain the following result.
\begin{theorem}\label{thm:volterra-essential}
Let $0<p<\infty$ and
$g\in\mathcal D$. Suppose that $T_g$ is bounded on
$\mathscr H^p$. Then
	\begin{equation}\label{eq:volterra-essential}
		\|T_g\|_{\mathrm e}
		\asymp_p
		\lim_{\varepsilon\rightarrow0}b_{g,p}(\varepsilon).
	\end{equation}
	Moreover, the following conditions are equivalent:
	\begin{enumerate}[label=\textup{(\roman*)}]
		\item $T_g$ is compact on $\mathscr H^p$;
		\item the Carleson embedding $f\mapsto f_\chi(\sigma)$ from
		$\mathscr H^p$ into $L^p(\nu_g)$ is compact;
		\item $\mathfrak C_{g,p}(\varepsilon)\to0$ as $\varepsilon\rightarrow0$;
		\item $b_{g,p}(\varepsilon)\to0$ as $\varepsilon\rightarrow0$.
	\end{enumerate}
Here, 
$
d\nu_g(\chi,\sigma)
:=|g_\chi'(\sigma)|^2\sigma\,d\sigma\,dm_\infty(\chi).
$
\end{theorem}

\begin{proof}
We first prove \eqref{eq:volterra-essential}. For $h\in\mathscr H^p$, put
\[
 S_\varepsilon h(\chi,\tau)
 :=\left(\int_{\Gamma_\tau\cap\{0<\sigma<\varepsilon\}}
               |h_\chi'(s)|^2\,d\sigma\,dt\right)^{1/2}.
\]
By \eqref{eq:Calderon-area},
$\|S_\varepsilon h\|_{L_{\i}^p}\le C_p\|h\|_{\mathscr H^p}$.
For each fixed $h$, the dominated convergence theorem gives
$\|S_\varepsilon h\|_{L_{\i}^p}\to0$ as $\varepsilon\rightarrow0$.
Moreover, Minkowski's inequality in the cone integral gives
$|S_\varepsilon h-S_\varepsilon v|\le S_\varepsilon(h-v)$.
The finite-net argument as in Lemma~\ref{lem:tail-compact-embedding}
therefore shows that, for every compact linear map $K$ on $\mathscr H^p$,
\[
 \sup_{\|f\|_{\mathscr H^p}\le1}
 \|S_\varepsilon(Kf)\|_{L_{\i}^p}\rightarrow0.
\]
Since $(T_gf)_\chi'=f_\chi g_\chi'$, the quasi-triangle inequality
together with the preceding estimates yields
\[
 \lim_{\varepsilon\rightarrow0}b_{g,p}(\varepsilon)
 \le C_p\|T_g-K\|.
\]
Taking the infimum over compact $K$ gives
$\lim_{\varepsilon\rightarrow0}b_{g,p}(\varepsilon)
 \lesssim_p\|T_g\|_{\mathrm e}$.

For the reverse inequality, fix $0<\varepsilon<1$.
By \eqref{eq:translated-volterra-bound}, $T_{{V_{\varepsilon/2}}g}$ is
bounded on $\mathscr H^p$, with norm at most $C_p\|T_g\|$.
We apply the translation and cone comparison in the proof of
Lemma~\ref{lem:interior-compact} to the family $\{\mu_{g_\chi}\}$.
If $s=w+\varepsilon/2\in\Gamma_\tau$ and $\Ree s\ge\varepsilon$,
then $\Ree w\ge\varepsilon/2$ and
\[
 \mathbf1_{\{w+\varepsilon/2\in\Gamma_\tau\}}
 \le\sum_{\eta\in\{-\varepsilon/2,0,\varepsilon/2\}}
       \mathbf1_{\{w\in\Gamma_{\tau+\eta}\}}.
\]
Since $|\eta|\le \varepsilon/2<1/2$, we have
\[
\frac{1}{1+(\tau-\eta)^2}
\asymp
\frac{1}{1+\tau^2},
\]
uniformly in $\tau\in\mathbb R$ and
$\eta\in\{-\varepsilon/2,0,\varepsilon/2\}$.
Using this, together with
$h_\chi(w+\varepsilon/2)=({V_{\varepsilon/2}}h)_\chi(w)$ and
$g_\chi'(w+\varepsilon/2)=({V_{\varepsilon/2}}g)_\chi'(w)$,
and then applying \eqref{eq:Calderon-area}, we obtain
\begin{align*}
 &\left[\int_{\bS^\infty}\int_{\mathbb R}
 \left(\int_{\Gamma_\tau\cap\{\sigma\ge\varepsilon\}}
     |h_\chi(s)|^2|g_\chi'(s)|^2\,d\sigma\,dt\right)^{p/2}
     \frac{d\tau}{1+\tau^2}\,dm_\infty(\chi)\right]^{1/p}\\
 &\qquad\le C_p\|T_{{V_{\varepsilon/2}}g}({V_{\varepsilon/2}}h)\|_{\mathscr H^p}
 \le{C_p\|T_g\|}\|{V_{\varepsilon/2}}h\|_{\mathscr H^p}.
\end{align*}

For $\delta>0$, the operator $K_{\delta}=T_g{V_{\delta}}$ is compact by
Theorem~\ref{thm:translation-compact}.
Apply \eqref{eq:Calderon-area} to $T_g(I-{V_{\delta}})f$ and split the integral at $\sigma=\varepsilon$. For the boundary part, using the definition of $b_{g,p}(\varepsilon)$, and for the interior part, using the preceding estimate, we can establish
$$\|T_g(I-V_{\delta})f\|_{\mathscr{H}^p}
\lesssim_p b_{g,p}(\varepsilon)\|(I-V_{\delta})f\|_{\mathscr{H}^p}+
\|T_g\|\|V_{\varepsilon/2}(I-V_{\delta})f\|_{\mathscr{H}^p}.$$
Since $\|I-{V_{\delta}}\|\le C_p$ and
${V_{\varepsilon/2}}(I-{V_{\delta}})=(I-{V_{\delta}}){V_{\varepsilon/2}}$, we have
\[
 \|T_g-K_{\delta}\|
 \le C_p b_{g,p}(\varepsilon)
     +{C_p\|T_g\|}\|(I-{V_{\delta}}){V_{\varepsilon/2}}\|.
\]
The strong convergence ${V_{\delta}}\to I$ is uniform on compact subsets of
$\mathscr H^p$, because the operators $I-{V_{\delta}}$ are uniformly bounded.
Since ${V_{\varepsilon/2}}$ is compact,
$\|(I-{V_{\delta}}){V_{\varepsilon/2}}\|\to0$ as $\delta{\rightarrow}0$.
First letting $\delta{\rightarrow}0$ and then $\varepsilon{\rightarrow}0$ gives
$\|T_g\|_{\mathrm e}\lesssim_p
 \lim_{\varepsilon{\rightarrow}0}b_{g,p}(\varepsilon)$.
This proves \eqref{eq:volterra-essential} and the equivalence
\textup{(i)}$\Longleftrightarrow$\textup{(iv)}.

We next prove \textup{(ii)}$\Longleftrightarrow$\textup{(iii)}.
In the $L^p(\nu_g)$-norms below, $f_\chi$ denotes the function
$(\chi,\sigma)\mapsto f_\chi(\sigma)$.
Suppose that the Carleson embedding in \textup{(ii)} is compact.
Let $q=\min\{1,p\}$. Given $\rho>0$, there exist
$f_1,\ldots,f_N$ in the unit ball of $\mathscr H^p$ such that, for every
$\|f\|_{\mathscr H^p}\le1$,
\[
 \|f_\chi-(f_j)_\chi\|_{L^p(\nu_g)}^q<\rho
\]
for some $j\in\{1,\ldots,N\}$. For each fixed $j$, the dominated convergence theorem gives
\[
 \|\mathbf1_{\{0<\sigma<\varepsilon\}}(f_j)_\chi\|_{L^p(\nu_g)}^q
 \rightarrow0\quad(\varepsilon\rightarrow0).
\]
Hence, for sufficiently small $\varepsilon$,
\[
 \sup_{\|f\|_{\mathscr H^p}\le1}
 \|\mathbf1_{\{0<\sigma<\varepsilon\}}f_\chi\|_{L^p(\nu_g)}^q
 \le 2\rho.
\]
Since $\rho>0$ is arbitrary, the definition of $\mathfrak C_{g,p}(\varepsilon)$ gives
\[
 \mathfrak C_{g,p}(\varepsilon)^{2/p}
 =\sup_{\|f\|_{\mathscr H^p}\le1}
   \|\mathbf1_{\{0<\sigma<\varepsilon\}}f_\chi\|_{L^p(\nu_g)}
 \rightarrow0\quad(\varepsilon\rightarrow0).
\]
Thus \textup{(iii)} holds.

Conversely, suppose \textup{(iii)} holds. Fix $\varepsilon>0$.
For $\sigma=u+\varepsilon/2\ge\varepsilon$, we have
$u\ge\varepsilon/2$ and $\sigma\le2u$.
Applying \eqref{eq:bps-embedding} to ${V_{\varepsilon/2}}g$
and using \eqref{eq:translated-volterra-bound} gives
\begin{align}\label{interior-embedding}
 &\int_{\bS^\infty}\int_\varepsilon^\infty
 |h_\chi(\sigma)|^p|g_\chi'(\sigma)|^2\sigma\,d\sigma\,dm_\infty(\chi)\nonumber\\
 &\quad\le2\int_{\bS^\infty}\int_0^\infty
 |({V_{\varepsilon/2}}h)_\chi(u)|^p
 |({V_{\varepsilon/2}}g)_\chi'(u)|^2u\,du\,dm_\infty(\chi)\nonumber\\
 &\quad\le2C(g,p)^2\|{V_{\varepsilon/2}}h\|_{\mathscr H^p}^p.
\end{align}
Let $\{f_n\}$ be a bounded sequence in $\mathscr H^p$.
Compactness of ${V_{\varepsilon/2}}$ gives a subsequence, still denoted by $\{f_n\}$, for which
$\{{V_{\varepsilon/2}}f_n\}$ converges. Applying the preceding estimate to
$h=f_n-f_m$ shows that
$\mathbf1_{\{\sigma\ge\varepsilon\}}(f_n)_\chi$ is Cauchy in
$L^p(\nu_g)$. Hence the restricted embedding
$f\mapsto\mathbf1_{\{\sigma\ge\varepsilon\}}f_\chi$ is compact.
Moreover,
\[
 \sup_{\|f\|_{\mathscr H^p}\le1}
 \|\mathbf1_{\{0<\sigma<\varepsilon\}}f_\chi\|_{L^p(\nu_g)}
 =\mathfrak C_{g,p}(\varepsilon)^{2/p}\rightarrow0.
\]
Thus the embedding in \textup{(ii)} is a operator norm limit of compact
linear maps and is compact by Lemma~\ref{lem:uniform-compact-approx}.

It remains to prove \textup{(i)}$\Longleftrightarrow$\textup{(ii)}.
If $g$ is constant, the assertion is immediate. Assume henceforth that $g$ is not constant. Fubini's theorem together with the
Haar invariance gives
\begin{equation}\label{eq:carleson-cone-identity}
	\int_{\bS^\infty}\int_{\mathbb R}\int_{\Gamma_\tau}
	|h_\chi(s)|^p|g_\chi'(s)|^2\,d\sigma\,dt
	\frac{d\tau}{1+\tau^2}\,dm_\infty(\chi)
	=2\pi\|{h_\chi}\|_{L^p(\nu_g)}^p.
\end{equation}
Indeed, the Haar invariance makes
$\int_{\bS^\infty}|h_\chi(\sigma+it)|^p|g_\chi'(\sigma+it)|^2
\,dm_\infty(\chi)$ independent of $t$, and
\[
 \int_{\mathbb R}\int_{|t-\tau|<\sigma}dt\,
       \frac{d\tau}{1+\tau^2}=2\pi\sigma.
\]
Moreover, since $(T_gh)'=hg'$ and $T_gh(+\infty)=0$,
\eqref{eq:Littlewood-Paley} gives
\begin{equation}\label{eq:volterra-Littlewood-Paley}
	\|T_gh\|_{\mathscr H^p}^p
	\asymp_p
	\int_{\bS^\infty\times(0,\infty)}
	|(T_gh)_\chi(\sigma)|^{p-2}|h_\chi(\sigma)|^2
	\,d\nu_g(\chi,\sigma).
\end{equation}
In the case $p=2$, the above estimate gives
$$\|T_gh\|_{\mathscr H^2}^2
\asymp
\|h_{\chi}\|^2_{L^2(\nu_g)}.$$
Suppose now $p>2$. Then by \eqref{eq:volterra-Littlewood-Paley} and H\"older's inequality
with conjugate exponents $p/(p-2)$ and $p/2$, we obtain
\begin{align*}
 \|T_gh\|_{\mathscr H^p}^p
 &\lesssim_p
 \int |(T_gh)_\chi|^{p-2}|h_\chi|^2\,d\nu_g\\
 &\le
 \left(\int |(T_gh)_\chi|^p\,d\nu_g\right)^{(p-2)/p}
 \left(\int |h_\chi|^p\,d\nu_g\right)^{2/p}\\
 &\le C(g,p)^{2(p-2)/p}
 \|T_gh\|_{\mathscr H^p}^{p-2}
 \|h_\chi\|_{L^p(\nu_g)}^2,
\end{align*}
where the last inequality follows from \eqref{eq:bps-embedding}
applied to $T_gh\in\mathscr H^p$. Hence
\[
\|T_gh\|_{\mathscr H^p}
\lesssim_{p,g}\|h_\chi\|_{L^p(\nu_g)}.
\]
Conversely, note that
\[
|h_\chi(s)|^p
\le N_{1,1}h(\chi,\tau)^{p-2}|h_\chi(s)|^2,
\qquad s\in\Gamma_\tau.
\]
Thus, by \eqref{eq:carleson-cone-identity}, H\"older's inequality, Lemma~\ref{lem:cone-maximal} and \eqref{eq:Calderon-area},
\[
\|{h_\chi}\|_{L^p(\nu_g)}^p
\lesssim_p
\|h\|_{\mathscr H^p}^{p-2}
\|T_gh\|_{\mathscr H^p}^2.
\]
Suppose next $0<p<2$. Since
\[
|h_\chi(s)|^2
\le N_{1,1}h(\chi,\tau)^{2-p}|h_\chi(s)|^p,
\qquad s\in\Gamma_\tau,
\]
the estimate \eqref{eq:Calderon-area} together with H\"older's inequality with exponents $2/(2-p)$ and $2/p$, Lemma~\ref{lem:cone-maximal} and
\eqref{eq:carleson-cone-identity} gives
\[
\|T_gh\|_{\mathscr H^p}^p
\lesssim_p
\|h\|_{\mathscr H^p}^{p(2-p)/2}
\|{h_\chi}\|_{L^p(\nu_g)}^{p^2/2}.
\]
In the other direction, H\"older's inequality,
\eqref{eq:volterra-Littlewood-Paley}, and
\eqref{eq:bps-embedding} applied to $T_gh$ yield
\begin{align*}
	\|{h_\chi}\|_{L^p(\nu_g)}^p
	\le
	\left(
	\int |(T_gh)_\chi|^{p-2}|h_\chi|^2\,d\nu_g
	\right)^{p/2}
	\left(
	\int |(T_gh)_\chi|^p\,d\nu_g
	\right)^{(2-p)/2} \lesssim_{p,g}
	\|T_gh\|_{\mathscr H^p}^p.
\end{align*}
These estimates, applied to differences in a bounded sequence $\{f_n\}\subset\mathscr H^p$, show that
\[
\|{(f_n-f_m)_\chi}\|_{L^p(\nu_g)}\to0
\quad\Longleftrightarrow\quad
\|T_g(f_n-f_m)\|_{\mathscr H^p}\to0.
\]
Therefore, $T_g$ is compact if and only if {the Carleson
embedding in \textup{(ii)}} is compact, proving the equivalence
\textup{(i)}$\Longleftrightarrow$\textup{(ii)}. The proof is complete.
\end{proof}

For $p=2$, the essential norm estimate in Theorem~\ref{thm:volterra-essential}
has the following exact form.

\begin{corollary}\label{cor:volterra-essential-exact-two}
Let $g\in\mathcal D$ and suppose that $T_g$ is bounded on $\mathscr H^2$.
Then
\begin{equation}\label{eq:volterra-essential-exact-two}
 \|T_g\|_{\mathrm e}
 =2\lim_{\varepsilon{\rightarrow}0}\mathfrak C_{g,2}(\varepsilon)
 =\sqrt{\frac2\pi}\,
   \lim_{\varepsilon{\rightarrow}0}b_{g,2}(\varepsilon).
\end{equation}
\end{corollary}

\begin{proof}
For $h(s)=\sum_{n\ge1}a_nn^{-s}\in\mathscr H^2$, orthogonality together with the equality
$$(\log n)^2\int_0^\infty e^{-2\sigma\log n}\sigma\,d\sigma
=\frac{1}{4}$$
for $n\ge2$ gives
\begin{equation}\label{h2-identity}
4\int_{\bS^\infty}\int_0^\infty
|h_\chi'(\sigma)|^2\sigma\,d\sigma\,dm_\infty(\chi)
=\|h\|_{\mathscr H^2}^2-|h(+\infty)|^2.
\end{equation}
In particular, the square root of the integral restricted to
$0<\sigma<\varepsilon$ is at most $\frac12\|h\|_{\mathscr H^2}$.
For every compact linear map $K:\mathscr H^2\to\mathscr H^2$, the dominated convergence theorem combined with the finite-net argument as in
Lemma~\ref{lem:tail-compact-embedding} therefore gives
\[
 \sup_{\|f\|_{\mathscr H^2}\le1}
 \left(\int_{\bS^\infty}\int_0^\varepsilon
 |(Kf)_\chi'(\sigma)|^2\sigma\,d\sigma\,dm_\infty(\chi)\right)^{1/2}
 \rightarrow0.
\]
Since $(T_gf)_\chi'=f_\chi g_\chi'$, the triangle inequality in $L^2$ gives
$2\lim_{\varepsilon\rightarrow0}\mathfrak C_{g,2}(\varepsilon)
\le\|T_g-K\|$.
Taking the infimum over compact $K$ gives
\[
 2\lim_{\varepsilon\rightarrow0}\mathfrak C_{g,2}(\varepsilon)
 \le\|T_g\|_{\mathrm e}.
\]

For the reverse inequality, let $K_{\delta}=T_gV_{\delta}$, $\delta>0$, as in the proof of
Theorem~\ref{thm:volterra-essential}. These operators are compact.
Since $V_{\delta}(n^{-s})=n^{-\delta}n^{-s}$ and $0\le1-n^{-\delta}\le1$,
orthogonality gives $\|I-V_{\delta}\|\le1$. For any $f\in\mathscr{H}^2$, apply the identity \eqref{h2-identity} to $T_g(I-V_{\delta})f$ and split the inner integral at $\sigma=\varepsilon$. For the boundary part, using the definition of $\mathfrak C_{g,2}(\varepsilon)$, and for the interior part, using the interior embedding estimate as in \eqref{interior-embedding}, we establish
$$\|T_g(I-V_{\delta})f\|_{\mathscr{H}^2}^2
\leq4\mathfrak C_{g,2}(\varepsilon)^2
\|(I-V_{\delta})f\|_{\mathscr H^2}^2+
8C(g,2)^2\|V_{\varepsilon/2}(I-V_{\delta})f\|_{\mathscr{H}^2}^2.$$
Since $V_{\varepsilon/2}$ commutes with $V_{\delta}$, these estimates give
\[
 \|T_g-K_{\delta}\|^2
 \le4\mathfrak C_{g,2}(\varepsilon)^2
   +8C(g,2)^2\|(I-V_{\delta})V_{\varepsilon/2}\|^2.
\]
For fixed $\varepsilon>0$, $\|(I-V_{\delta})V_{\varepsilon/2}\|\to0$ as
$\delta\rightarrow0$, because $V_{\varepsilon/2}$ is compact and
$V_{\delta}\rightarrow I$ strongly. First letting $\delta\rightarrow0$ and then
$\varepsilon\rightarrow0$ proves the reverse inequality, and hence
$\|T_g\|_{\mathrm e}
=2\lim_{\varepsilon\rightarrow0}\mathfrak C_{g,2}(\varepsilon)$.
Finally, \eqref{eq:carleson-cone-identity} restricted to
$0<\sigma<\varepsilon$ gives
$b_{g,2}(\varepsilon)^2=2\pi\mathfrak C_{g,2}(\varepsilon)^2$,
which proves \eqref{eq:volterra-essential-exact-two}.
\end{proof}

The $\operatorname{VMOA}$ result now follows immediately from the uniform boundary vanishing condition established above.

\begin{corollary}\label{cor:volterra-vmoa}
	Let $0<p<\infty$ and
	$g\in\mathcal D\cap\operatorname{VMOA}(\mathbb C_0)$. Then $T_g$ is compact
	on $\mathscr H^p$.
\end{corollary}

\begin{proof}
The	boundedness of $T_g$ follows from \cite[Theorem~5.3]{BPSVolterra}. For
	$0<\varepsilon<1$, put
	\[
	\omega(\varepsilon)
	:=\sup_{\chi\in\bS^\infty}
	\|\mu_{g_\chi}|_{\{0<\sigma<\varepsilon\}}\|_{CM}.
	\]
	Lemma~\ref{lem:uniform-vmoa-vertical} gives $\omega(\varepsilon)\to0$ as $\varepsilon\to0$. Arguing as in the proof of \cite[Theorem 4.1]{ChenWangYuan}, we can establish that for any $f\in\mathscr{H}^p$,
	$$\int_{\bS^{\infty}}\int_{\mathbb{R}}
	\left(\int_{\Gamma_{\tau}\cap\{0<\sigma<\varepsilon\}}|f_{\chi}(s)|^2|g'_{\chi}(s)|^2\,d\sigma\,dt\right)^{p/2}\frac{d\tau}{1+\tau^2}dm_{\infty}(\chi)
	\leq C_p\omega(\varepsilon)^{p/2}\|f\|_{\mathscr{H}^p}^p,$$
	which implies that
	$$b_{g,p}(\varepsilon)\le C_p^{1/p}\omega(\varepsilon)^{1/2}\rightarrow0$$
	as $\varepsilon\to0$. Therefore, the desired compactness follows from Theorem \ref{thm:volterra-essential}.
\end{proof}

\end{document}